\documentclass[a4paper]{amsart}
\usepackage{a4}
\usepackage{amsmath}
\usepackage{amssymb}
\usepackage[applemac]{inputenc}
\usepackage{verbatim}
\usepackage{soul}
\usepackage{graphicx}
\usepackage{cancel}
\usepackage{xcolor}
\usepackage{color}
\usepackage{enumerate}
\usepackage{subcaption}
\usepackage{comment}
\usepackage{todonotes}
\def\dd{{\rm d}}
 \newtheorem{theo}{Theorem}[section]
 \newtheorem{lem}{Lemma}[section]
 \newtheorem{prop}{Proposition}[section]

 \newtheorem{dfn}{Definition}[section]
 \newtheorem{rem}{Remark}[section]
  \newtheorem{example}{Example}[section]
\newcommand{\be}{\begin{eqnarray}}
\newcommand{\ee}{  \end{eqnarray}}
\newcommand{\beno}{\begin{eqnarray*}}
\newcommand{\eeno}{  \end{eqnarray*}}

\newcommand\R{{\mathbb R}}
\newcommand\A{{\mathcal A}}
\newcommand\FF{{\mathbb F}}
\newcommand\EE{{\mathbb E}}
\newcommand\PP{{\mathbb P}}
\newcommand\F{{\mathcal F}}
 \newcommand\I{{\mathcal I}}
\newcommand{\vare}{\varepsilon}
\newcommand{\NN}{\mathbb{N}}

\newcommand{\Dx}{{\Delta x}}
\newcommand{\Dt}{{\Delta t}}

\author[Elisabetta Carlini]{Elisabetta Carlini\ec}
\thanks{\ec Dipartimento di Matematica Guido Castelnuovo, ``Sapienza'', Universit\`a di Roma, 00185 Rome,
Italy (carlini@mat.uniroma1.it).}
\author[Athena Picarelli]{Athena Picarelli\ap}
\thanks{\ap Dipartimento di Scienze Economiche, Universit\`a degli Studi di Verona, 37129 Verona, Italy (athena.picarelli@univr.it).}
\author[Francisco J. Silva]{Francisco J. Silva\fs} 
\thanks{\fs Institut de recherche XLIM-DMI, UMR-CNRS 7252, Facult\'e des Sciences et Techniques, 
Universit\'e de Limoges, 87060 Limoges, France (francisco.silva@unilim.fr).}
\newcommand{\ec}{$^1$} \newcommand{\ap}{$^2$}\newcommand{\fs}{$^3$}
\begin{document} 
\title[A SL scheme for HJB equations with Dirichlet boundary conditions]{A Semi-Lagrangian scheme for Hamilton-Jacobi-Bellman equations with Dirichlet boundary conditions} 

\begin{abstract}
We study the numerical approximation of time-dependent, possibly degenerate, second-order Hamilton-Jacobi-Bellman equations in bounded domains with nonhomogeneous Dirichlet boundary conditions. It is well known that convergence towards the exact solution of the equation, considered here in the viscosity sense, holds if the scheme is monotone, consistent, and stable. While standard finite difference schemes are, in general, not monotone, the so-called semi-Lagrangian schemes are monotone by construction. On the other hand, these schemes make use of a wide stencil and, when the equation is set in a bounded domain, this typically causes an overstepping of the boundary and hence the loss of consistency. We propose here a semi-Lagrangian scheme defined on an unstructured mesh, with a suitable treatment at grid points near the boundary to preserve consistency, and show its convergence for problems where the viscosity solution can even be discontinuous. We illustrate the numerical convergence in several tests, including degenerate and first-order equations.
\end{abstract}
\date{\today} 

\maketitle 
{\small
\noindent {\bf AMS-Subject Classification:} 49L25, 65M12, 35D40, 35K55.  \\[0.5ex]
\noindent {\bf Keywords:} Hamilton-Jacobi-Bellman equations, Dirichlet boundary conditions, semi-Lagrangian schemes, convergence results, numerical experiences.
}

\section{Introduction}\label{sec:Intro}

We consider the following second-order, possibly degenerate, Hamilton-Jacobi-Bellman (HJB) equation
\begin{subequations}
\label{eq:HJB}
\begin{align}
 -\partial_t v+\max_{a\in A} \mathcal{L}^{a}(t, x,D_x v, D^2_{x} v)&=0\qquad\text{in } Q_T, 
 \label{parabolic_HJB} \\
 v&= \Psi\qquad\text{on }\partial^* Q_T,
 \label{boundary_conditions}
\end{align}
\end{subequations}
where, for a given $T>0$ and a nonempty open and bounded domain $\Omega \subset \R^d$, $Q_T:=[0, T)\times\Omega$ and $\partial^* Q_T := (\{T\}\times  \overline{\Omega}) \cup ([0,T)\times \partial\Omega)$ denote, respectively, the parabolic domain and its parabolic boundary,  
$\emptyset\neq A\subset\R^m$ is compact, $\Psi\colon\partial^{*}Q_T\to\R$ and, given $f\colon\overline{Q}_T\times A\to\R$, $b\colon\overline{Q}_T\times A\to\R^d$, and $\sigma\colon\overline{Q}_T\times A\to\R^{d\times p}$, for every $a\in A$, 
\begin{equation} 
\label{eq:def_L}
\mathcal{L}^{a}(t,x,q,P)=-\frac{1}{2}\text{Tr}[\sigma \sigma^\top(t,x,a) P]
-b(t,x,a)^{\top}q-f(t,x,a)\quad\text{for all }(t,x,q,P)\in \overline{Q}_T\times\R^{d}\times\R^{d\times d}.
\end{equation}

Observe that if $A=\{a\}$ for some $a\in\R^{m}$, by considering a simple change of the time variable, equation~\eqref{eq:HJB} reduces to a standard, possibly degenerated, linear parabolic equation. When $A$ is not a singleton, system~\eqref{eq:HJB} is nonlinear and naturally arises in the study of the optimal control of diffusion processes (see, e.g.,~\cite{FS06}). 

Under several assumptions on the data, imposing a kind of compatibility condition with the boundary value $\Psi$, the existence and uniqueness of a continuous viscosity solution to~\eqref{eq:HJB} (see, e.g.,~\cite{CIL92}) can be found, for instance, in~\cite{MR852356} and~\cite{FS06} for first- and second-order equations, respectively. For more general data, the study of the existence and uniqueness of viscosity solutions satisfying a weak form of the boundary condition has been conducted in~\cite{BP88,I89,MR1014943} for first-order equations, in~\cite{BB95,BR98} for stationary second-order equations, and in~\cite{MR2111030} for a class of time-dependent problems. 

The numerical approximation of boundary value problems such as~\eqref{eq:HJB} is practically relevant not only when the original problem is posed on a bounded domain, but also when it is defined on the whole space and the localization to a bounded region is required for computational tractability. On the other hand, its numerical analysis is significantly more challenging than the one where $\Omega=\R^{d}$, as the regularity of the solution near the boundary $\partial \Omega$ plays an important role. In this context, Godunov-type schemes for first-order stationary and time-dependent HJB equations with nonhomogeneous Dirichlet conditions have been considered in~\cite{Abgrall03} and~\cite{W08}, respectively. A variation of the abstract scheme in~\cite{Abgrall03} is investigated in~\cite{W08} and yields, in particular, the convergence of a finite volume approximation of a time-dependent HJB equation with nonhomogeneous Dirichlet conditions. For second-order stationary HJB equations with homogeneous Dirichlet conditions satisfying the so-called Cordes condition, a convergent discontinuous Galerkin finite element approximation is studied in~\cite{MR3077903}. In the case of time-dependent HJB equations with  Dirichlet conditions and admitting continuous solutions,~\cite[Chapter IX]{FS06} and~\cite{MR3033005} show the convergence of monotone finite difference  and finite element methods, respectively. Let us also mention the work~\cite{MR4699587} dealing with finite element approximation of time-dependent HJB equations with mixed boundary conditions and~\cite{MR3823876}, which discusses the relevance of the definition of boundary conditions for HJB equations including several numerical examples. For a more complete overview on  existing numerical methods to approximate solutions to fully nonlinear elliptic and parabolic equations with or without boundary conditions, we refer the reader to~\cite{MR3049920,MR3653852} and the references therein.

In this work, we study the approximation of~\eqref{eq:HJB} through semi-Lagrangian (SL) schemes, which are traditionally based on the discretization of representation formulae of solutions to PDEs in terms of their characteristic curves. Compared with standard finite difference schemes, one advantage of SL schemes is that they are at the same time explicit and stable under an inverse Courant-Friedrichs-Lewy (CFL) condition and hence allow for large time steps. We refer the reader to~\cite{FalcFerrBook} for an overview of SL schemes for first-order HJB equations. When considering equation~\eqref{eq:HJB} with $\Omega=\R^{d}$, in \cite{Men89,CamFal95,CFF10,DebrJako13} the authors consider a SL discretization that takes advantage of the fact that the viscosity solution is given by the so-called {\it value function} of an optimal control problem of a diffusion process evolving in $\R^{d}$. Roughly speaking, they consider a stochastic Euler time-discrete scheme for the underlying controlled characteristic curves combined with interpolation methods to deal with the space variable. Similarly, as recalled in Section~\ref{preliminaries_parabolic_HJB} below, the solution to~\eqref{eq:HJB} is given by the value function of an optimal control problem of a controlled diffusion where a cost, represented by $\Psi(t,x)$, is incurred as soon as the controlled trajectory hits the boundary at point $x\in\partial\Omega$ at time $t\in [0,T]$, and the trajectory is stopped at $x$ until the final time $T$. If one proceeds as in the case where $\Omega=\R^d$, a particular difficulty arises at grid points near the boundary $\partial\Omega$, for which at least one of the possible time-discrete trajectories hits $\partial \Omega$ before the prescribed time step. Indeed, for those points, the SL discretization in~\cite{ Men89,CamFal95,DebrJako13} looses its consistency. In order to cope with this problem, some extensions of SL schemes for PDEs with Dirichlet boundary conditions have been considered in the literature. In~\cite{FJ17, Jensen_non_convex}, a SL scheme for the Monge-Amp\`ere equation, based on a rescaling of the space mesh near the boundary, is presented. A second-order fully discrete SL scheme, based on an extrapolation technique, is considered in~\cite{BCCF2020} to solve systems of advection-diffusion-reaction equations. Finally, a  SL scheme for second-order HJB equations has been proposed in~\cite{ReisRota17, PicaReisRota18}, for which, in order to deal with some instabilities, an implicit version had to be introduced.

Inspired by the scheme proposed by Milstein and Tretyakov in~\cite{MR1867423} (see also~\cite{BonaFerrRocc18,BF2014} for advection-diffusion equations arising in fluid dynamics), which deals with semi-linear equations, we consider a SL scheme that truncates discrete characteristic curves as soon as they hit the boundary $\partial\Omega$ and balances their weights in order to preserve consistency. In the case of linear equations, our scheme coincides with the one in~\cite{MR1867423} for one-dimensional problems ($d=1$) but differs from it in higher dimensions. For the general HJB equation~\eqref{eq:HJB}, the scheme considered here is explicit, consistent, monotone, stable, and provides a natural extension of the one in~\cite{ Men89,CamFal95,DebrJako13} to bounded domains. Our main results are contained in Theorem~\ref{theo:main_result_1} and in Theorem~\ref{theo:main_result_2}, which are both established by using the half-relaxed limit technique introduced in~\cite{BS91}. In Theorem~\ref{theo:main_result_1}, which deals with the case where the solution to~\eqref{eq:HJB} is continuous, we show the uniform convergence in $\overline{Q}_{T}$ of solutions of the scheme, while in Theorem~\ref{theo:main_result_2}, which covers the case of discontinuous solutions to~\eqref{eq:HJB}, the convergence is uniform over compact subsets of $[0,T]\times \Omega$. Notice that both convergence results hold under the same conditions over the time and discrete step than those in the standard case $\Omega=\R^{d}$.

The paper is structured as follows.  In the next section, we state our main assumptions, recall the probabilistic interpretation of~\eqref{eq:HJB} and the notions of viscosity solutions satisfying strong and weak Dirichlet boundary conditions. Section~\ref{sec:milstein} recalls the scheme in~\cite{MR1867423} in the case of a one-dimensional ($d=1$) linear parabolic equation and provides the construction of a consistent multidimensional extension. Based on the latter, we introduce in Section~\ref{sec:fullydiscrete} the SL scheme for the HJB equation~\eqref{eq:HJB} and show our main convergence results in Theorems~\ref{theo:main_result_1} and~\ref{theo:main_result_2}. Finally, we illustrate in Section~\ref{sec:tests} the numerical convergence of the scheme in several tests including non-degenerate, degenerate, and also first-order HJB equations. 
\smallskip

{\bf Acknowledgements.}  E. Carlini was partially supported by European Union - NextGenerationEU - and by the Italian Ministry of Instruction, University and Research (MIUR), project title: Optimal control problems: analysis, approximation - PRIN Project2022238YY5, by INdAM-research group GNCS, project title: Metodi numerici avanzati per equazioni alle derivate parziali in fenomeni di trasporto e diffusione - CUP E53C24001950001, and  by KAUST through the subaward agreement ORA-2021-CRG10-4674.6. 
A. Picarelli acknowledges financial support under the National Recovery and Resilience Plan (NRRP), Mission 4, Component 2, Investment 1.1, Call for tender No.
1409 published on 14.9.2022 by the Italian Ministry of Universities and Research (MUR), funded
by the European Union -  NextGenerationEU - project title: Probabilistic methods for energy
transition - CUP G53D23006840001 - Grant Assignment Decree No. 1379 adopted on 01/09/2023
by MUR. F. J. Silva was partially supported by l'Agence Nationale de la Recherche (ANR), project  ANR-22-CE40-0010, and by KAUST through the subaward agreement ORA-2021-CRG10-4674.6.

\section{Preliminaries on the parabolic HJB equation}
\label{preliminaries_parabolic_HJB}

In this section, we recall the probabilistic interpretation of~\eqref{eq:HJB}, which forms the basis of our SL scheme, and motivate the notions of viscosity solution to~\eqref{eq:HJB} for both strong and weak boundary conditions.

In what follows, we will consider the following assumption:  
\smallskip
\begin{enumerate}
\item[\bf(H1)] The functions $\Psi$, $f$, $b$, and $\sigma$ are continuous. Moreover, for $\phi=b$, $\sigma_{ij}$ $(i=1,\hdots,d,\,j=1,\hdots,p)$ one has the following Lipschitz property: there exists $L>0$ such that
\begin{equation}
\big|\phi(t,x,a)-\phi(t,y,a)\big|\leq L|x-y|\quad\text{for all }t\in [0,T],\,x,\,
y\in \overline{\Omega},\,a\in A.
\end{equation}
\end{enumerate}

Let $(\Omega,\F,\PP)$ be a probability space carrying a $p$-dimensional Brownian motion $W$ in $[0,T]$. We consider the natural filtration
$\FF:=(\F_t)_{0 \leq t\leq T}$, where $\F_t$ is the $\sigma$-field generated by 
$\{W(s)\,|\,0\leq s\leq t\}$ completed with the $\PP$-null sets. Let $\A$ be the set of $\FF$-progressively measurable $A$-valued control processes (see, e.g.,~\cite[Definition 1.11]{MR1121940}) and, given $(t,x)\in\overline{Q}_T$ and $\alpha \in \A$, let $Y^{t,x,\alpha}$ be the unique solution to 
\begin{equation}
\label{controlled_eq}
\begin{aligned}
\dd Y(s)&=b(s,Y(s),\alpha(s))\dd s+\sigma(s,Y(s),\alpha(s))\dd W(s)\quad\text{for }s\in(t,T),\\
Y(t)&=x.
\end{aligned}
\end{equation}
It follows from~{\bf(H1)} that $Y^{t,x,\alpha}$ is well-defined for every $\alpha\in\A$ (see, e.g.,~\cite[Chapter 1, Theorem 6.16]{YZ99}). Define the {\it stopping time} 
\begin{equation}
\tau^{t,x,\alpha}:=\inf\{s\in [t,T]\,|\,Y^{t,x,\alpha}(s)\notin\Omega\},
\end{equation} 
with the convention $\inf\emptyset=+\infty$, and consider the {\it value function} $V\colon\overline{Q}_{T}\to\R$, given by 
\begin{equation}
\label{value_function_exit_from_open}
V(t,x)=\inf_{\alpha\in\A}\EE\Bigg(\int_{t}^{T\wedge\tau^{t,x,\alpha}}
f\left(s,Y^{t,x,\alpha}(s),\alpha(s)\right)\dd s+
\Psi\big(T\wedge\tau^{t,x,\alpha},Y^{t,x,\alpha}(T\wedge\tau^{t,x,\alpha})\big)\Bigg)
\end{equation}
for all $(t,x) \in \overline{Q}_T$. Using  {\bf(H1)} together with the dynamic programming principle (see, for instance, \cite[Chapter V, Section 2]{FS06}), $V$  can be shown to be a viscosity solution to \eqref{parabolic_HJB} in the following sense (see e.g. \cite[Chapter VII, Definition 4.2 and Remark 4.1]{FS06}). 
\begin{dfn}
\label{def_viscosity_solution_Q_T}
 A locally bounded upper {\rm(}resp. lower{\rm)} semicontinuous function $v\colon\overline Q_T\to\R$ is a viscosity subsolution {\rm(}resp. supersolution{\rm)} to \eqref{parabolic_HJB} if for every $(t,x)\in Q_T$ and $\phi\in C^{\infty}(Q_T)$
such that $v-\phi$ has local maximum {\rm(}resp. local minimum{\rm)} at $(t,x)$, one has
\begin{align}
-\partial_{t}\phi(t,x)+\max_{a\in A}\mathcal{L}^{a}(t,x,D_x\phi(t,x),D_x^2\phi(t,x))\leq 0,\nonumber\\ 
\Big(\text{resp. }-\partial_{t}\phi(t,x)+
\max_{a\in A}\mathcal{L}^{a}(t,x,D_x\phi(t,x),D_x^2\phi(t,x))\geq 0\Big).
\end{align}

A locally bounded function $v\colon\overline Q_T\to\R$ is a viscosity solution to
\eqref{parabolic_HJB} if $v^{*}\colon\overline{Q}_T\to\R$ and 
$v_{*}\colon\overline{Q}_T\to\R$, defined by
\begin{equation}
v^{*}(t,x):=\limsup_{(s,y)\to (t,x)}v(s,y)\quad\mbox{and}\quad 
v_{*}(t,x):=\liminf_{(s,y)\to (t,x)}v(s,y)\quad\text{for all }(t,x)\in\overline{Q}_T,
\end{equation}
are, respectively,   sub- and supersolutions to \eqref{parabolic_HJB}.
\end{dfn}

Concerning the boundary condition~\eqref{boundary_conditions}, it follows from~\eqref{value_function_exit_from_open} that $V(t,x)= \Psi(t,x)$ for all $(t,x) \in \partial^* Q_T$. Thus, if $V\in C(\overline{Q}_T)$, then $V=V^{*}=V_{*}$ in $\overline{Q}_T$ and hence, in particular, $V^{*}=V_{*}=\Psi$ in $\partial^* Q_T$. This suggests the following definition.

\begin{dfn}
\label{strong_boundary_conditions} 
\begin{enumerate}[{\rm(i)}]
\item A locally bounded upper {\rm(}resp. lower{\rm)} semicontinuous function
$v\colon\overline Q_T\to\R$ is a viscosity subsolution {\rm(}resp. supersolution{\rm)} to~\eqref{eq:HJB} in the strong sense if it is a subsolution {\rm(}resp. supersolution{\rm)} to~\eqref{parabolic_HJB} and, for every
$(t,x)\in\partial^{*}Q_T$, we have $v(t,x)\leq\Psi(t,x)$ {\rm(}resp.
$v(t,x)\geq\Psi(t,x)${\rm)}. 
\item A locally bounded function $v\colon\overline{Q}_T\to\R$ is a viscosity solution to~\eqref{eq:HJB} in the strong sense if $v^{*}$ and $v_{*}$ are sub- and supersolutions to~\eqref{eq:HJB} in the strong sense, respectively.
\end{enumerate}
\end{dfn}

It follows from the previous discussion that if the value function $V$ is continuous in $\overline{Q}_T$, then it is a viscosity solution to~\eqref{eq:HJB} in the strong sense. Moreover, in this case, the following comparison result shows that $V$ is the unique viscosity solution to~\eqref{eq:HJB} in the strong sense.

\begin{prop}
\label{prop:weak_comparison} 
Assume that {\bf(H1)} hold and that $\partial \Omega$ is of class $C^2$. 
Let $v_{1}\colon\overline{Q}_{T}\to\R$ be upper semicontinuous and let $v_{2}\colon\overline{Q}_{T}\to\R$ be lower semicontinuous. Suppose that $v_1$ and $v_2$ are, respectively, viscosity sub- and supersolutions to~\eqref{eq:HJB} in the strong sense. Then we have $v_{1}\leq v_{2}$ in $\overline{Q}_{T}$.
\end{prop}
\begin{proof}
See e.g.~\cite[Chapter V, Theorem 8.1]{FS06}.
\end{proof}

Given $\beta\in (0,1]$, let us define
\begin{equation}
\mathcal{F}_{\beta}(\overline{Q}_T):=\left\{\varphi\in C^{1,2}(\overline{Q}_{T})\, \bigg|\,\sup_{x\in\overline{\Omega},\,(s,t)\in [0,T]^2,\,s\neq t}
\frac{|D_{x}\varphi(t,x)-D_{x}\varphi(s,x)|}{|t-s|^{\beta}}<+\infty \right\} 
\end{equation}
and, for every $\delta>0$, let us set $\Omega_{\delta}:=\{x\in\Omega\,|\,\text{dist}(x,\partial\Omega)<\delta\}$, where, for every $\mathfrak{C}\subset\R^{d}$,  $\text{dist}(x,\mathfrak{C})=\inf_{y\in \mathfrak{C}}|y-x|$ is the distance from $x$ to $\mathfrak{C}$. 

The following assumption, which assumes the existence of a barrier function 
and that the boundary data $\Psi$ can be extended to a smooth subsolution to~\eqref{eq:HJB}, implies the continuity of $V$ on $\overline{Q}_{T}$.
\vspace{0.2cm}

{\bf(H2)} The following hold: 
\begin{enumerate}[{\rm(i)}]
\item There exists $\delta>0$, $\zeta\in C^2(\overline{\Omega}_\delta)$ such that $\zeta=0$ in $\partial\Omega$, $\zeta>0$ 
in $\overline\Omega_{\delta}\setminus\partial\Omega$, and there exists $\eta>0$ such that, for every $(t,x)\in [0,T)\times\Omega_\delta$, there exists $a\in A$ for which the following holds:
\begin{equation}
\label{super_solution_barrier}
-\frac{1}{2}\text{Tr}\big(\sigma \sigma^\top(t,x,a) D^{2}_{x}\zeta(x)\big)
- b(t,x,a)^{\top}D_x\zeta(x)> \eta.
\end{equation}
\item There exists $\beta\in (\frac{1}{2},1]$ and $g\in\mathcal{F}_{\beta}(\overline{Q}_{T})$ such that $g\leq\Psi$ in $\partial ^{*}Q_{T}$, $g=\Psi$ in 
$[0,T] \times \partial\Omega$, and 
\begin{equation}
\label{eq:g_bc}
-\partial_{t}g(t,x)+\sup_{a\in A}\mathcal{L}^{a}(t,x,D_{x}g(t,x), D_{x}^{2}g(t,x))
\leq 0\quad\text{for all }(t,x)\in\overline{Q}_{T}.
\end{equation}
\end{enumerate}
\begin{prop} 
\label{prop:value_function_unique_viscosity_sol}
Assume that {\bf(H1)}-{\bf(H2)} hold and that $\partial \Omega$ is of class $C^2$. Then $V\in C(\overline{Q}_{T})$. In particular, $V$ is the unique continuous viscosity solution to~\eqref{eq:HJB} in the strong sense. 
\end{prop}
\begin{proof}
The first result follows from~\cite[Chapter V, Theorem 2.1 and Chapter IX, Remark 5.1]{FS06}, while the second one follows from  Proposition~\ref{prop:weak_comparison}.
\end{proof}
\begin{rem} 
{\rm(i)} The result in~\cite[Chapter V, Theorem 2.1]{FS06} shows that the
continuity of $V$ also holds if {\bf(H2)} is satisfied with $\eta=0$ in~\eqref{super_solution_barrier}. See also~\cite{MR2763549} for a continuity result of the value function under more general assumptions. On the other hand, and as
in~\cite[Chaper IX, Theorem 5.2]{FS06}, the strict positivity of $\eta$ in \eqref{super_solution_barrier} will play an important role in our convergence result in Theorem~\ref{theo:main_result_1}. We refer the reader to~\cite[Chapter IX, Section 5]{FS06} for some examples of data satisfying {\bf(H2)(i)} {\rm(}including the case of non-degenerate second-order operators{\rm)}. \medskip\\
{\rm(ii)} {\rm(Homogeneous boundary condition)} If $\Psi=0$ in $[0,T]\times\partial\Omega$, $f\geq 0$, and $\Psi(T,\cdot)\geq 0$, then {\bf(H2)}{\rm(ii)} holds with $g\equiv 0$.\medskip\\
{\rm(iii)} {\rm(Smooth solutions)}  If~\eqref{eq:HJB} admits a classical solution $v\in\F_{\beta}(\overline{Q}_{T})$, for some $\beta\in(\frac{1}{2},1]$, then {\bf(H2)}{\rm(ii)} automatically holds with $g=v$.  
\end{rem}

As the following example shows, if {\bf(H2)} is not satisfied, the value function can fail to be continuous. 
\begin{example}\label{ex1}
Suppose that $d=1$, $\Omega=(0,1)$, $T=1$, $m=1$, $A=[-1,1]$, and define $\Psi\colon\partial^*Q_{1}\to\R$ as $\Psi(t,x)=(1-t)(1-x)$ for all $(t,x)\in\partial^*Q_{1}$. We also set $b(t,x,a)=a$, $\sigma(t,x,a)=0$, and $f(t,x,a)=0$ for all $(t,x)\in\overline{Q}_{1}$ and $a\in A$. In this case, one checks that $V$, defined in~\eqref{value_function_exit_from_open}, is given by
\begin{equation}
\label{eq:explicit_V_example}
V(t,x)=
\begin{cases}
0 &\text{if }(t,x)\in[0,1]\times(0,1],\\
1-t&\text{if }(t,x)\in[0,1]\times\{0\},
\end{cases}
\end{equation}
which is discontinuous at $(t,0)$ for every $t\in[0,1)$. Notice that, in this case, {\bf(H2)}{\rm(i)} holds but {\bf(H2)}{\rm(ii)} does not as, if $g\in\F_{\beta}$ {\rm(}for some $\beta\in(\frac{1}{2},1]${\rm)} satisfies $g=\Psi$ in $\partial^*Q_{1}$ and 
$$
-\partial_t g(t,x) +\sup_{a\in[-1,1]}aD_{x}g(t,x) \leq 0\quad\text{for all }(t,x)\in\overline{Q}_1,
$$
then, taking $a=0$ yields that, for every $x\in [0,1]$, $g(\cdot,x)$ is increasing. This contradicts the equalities $g(0,0)=\Psi(0,0)=1$ and $g(1,0)=\Psi(1,0)=0$. 
\end{example}

Note that, by Definition~\ref{strong_boundary_conditions},  if $V$ is not continuous in $[0,T) \times \partial \Omega$, then it cannot be a viscosity solution to~\eqref{eq:HJB} in the strong sense. However, in this case one can still show that $V$ solves~\eqref{eq:HJB} in a weaker sense, which can be motivated as follows. If $V^{*}(t,x)>\Psi(t,x)$  at  some $(t,x)\in [0,T) \times \partial \Omega$, it can be shown that for any $\phi \in C^{\infty}(\overline{Q}_T)$ such that $V^{*} - \phi$ has a local maximum at $(t,x)$ one has
$$ 
-\partial_t\phi(t,x)+\sup_{a\in A} \mathcal{L}^{a}(t, x,D_x\phi(t,x) , D_x^2\phi(t,x)) \leq 0$$
(see e.g.~\cite[Chapter 5]{Barles_livre_94} for a proof in the deterministic framework). Similarly, if at some $(t,x) \in [0,T) \times \partial \Omega$ the inequality $V_{*}(t,x) \geq \Psi(t,x)$ does not hold,  then, for any $\phi \in C^{\infty}(\overline{Q}_T)$ such that $V_{*} - \phi$ has a local minimum at $(t,x)$ one has
$$ - \partial_t\phi(t,x)+\sup_{a\in A} \mathcal{L}^{a}(t, x,D_x\phi(t,x) , D_x^2\phi(t,x)) \geq 0.$$

The previous discussion motivates the following definition.
\begin{dfn}
\label{weak_boundary_conditions}  
{\rm(i)} A locally bounded upper {\rm(}resp. lower{\rm)} semicontinuous function $v\colon\overline Q_T\to \R$ is a viscosity subsolution {\rm(}resp. supersolution{\rm)} to \eqref{eq:HJB} in the weak sense if it is a viscosity subsolution {\rm(}resp. supersolution{\rm)} to \eqref{parabolic_HJB} and the following properties hold: \smallskip\\
{\rm(i.1)}  for any $x\in \overline{\Omega}$ one has $v(T,x) \leq \Psi(T,x)$  {\rm(}resp. $v(T,x) \geq \Psi(T,x)${\rm)}; \smallskip\\
{\rm(i.2)} for any $(t,x)\in  \partial^*Q_T$ with $t<T$ and any $\phi \in C^{\infty}(\overline{Q}_T)$ such that $v-\phi$ has a local maximum {\rm(}resp. local minimum{\rm)} at $(t,x)$ one has
$$\begin{array}{l}
 \min  \Big(  -\partial_t \phi(t,x)+\sup_{a\in A} \mathcal{L}^{a}(t, x,D_x\phi(t,x) , D_x^2\phi(t,x)) , v - \Psi\Big)  \leq 0, \\[4pt]
 \Big(\text{resp. }    \max \Big(- \partial_t\phi(t,x)+\sup_{a\in A} \mathcal{L}^{a}(t, x,D_x\phi(t,x) , D_x^2\phi(t,x)),v-\Psi\Big)  \geq 0\Big).
\end{array}
$$
{\rm(ii)} A locally bounded function  $v\colon\overline Q_T\to \R$ is a viscosity solution to \eqref{eq:HJB} in the weak sense if $v^*$ and $v_{*}$ are sub- and supersolutions to \eqref{eq:HJB} in the weak sense, respectively.
\end{dfn}

The above notion of discontinuous viscosity solution was introduced, in the framework of first-order equations, independently by Barles and Perthame in \cite{BP88} and by Ishii in \cite{I89}.
As explained above, the value function $V$ solves~\eqref{eq:HJB} in the weak sense, which settles the question of existence of a solution. On the other hand, the uniqueness of a viscosity solution to \eqref{eq:HJB}  usually follows from a comparison principle between its weak sub- and supersolutions. 
\begin{dfn}
\label{strong_comparison_principle}
We say that a strong comparison principle holds for \eqref{eq:HJB} if for any sub- and supersolutions $v_1$ and $v_2$, respectively, of  \eqref{eq:HJB} in the weak sense, one has $v_1 \leq v_2$ in $[0,T]\times\Omega$. 
\end{dfn}
The comparison principle in  Definition \ref{strong_comparison_principle} is stronger than the one in Proposition~\ref{prop:weak_comparison}, used to establish the uniqueness of strong viscosity solutions. It plays an important role in the convergence analysis of the scheme that we propose when the value function $V$ may fail to be continuous in $[0,T]\times\partial\Omega$.  Assumptions on the data near the boundary ensuring a strong comparison principle are studied in \cite{BB95, BR98} for elliptic equations in smooth domains. Some extensions of these results to the case of nonsmooth domains and parabolic equations have been addressed in~\cite{MR2111030}. 

A straightforward consequence of the strong comparison principle is the uniqueness of viscosity solutions to~\eqref{eq:HJB} in the weak sense when restricted to $[0,T]\times\Omega$. 

\begin{prop}
\label{prop:uniqueness_sol_weak_sense} 
Assume that {\bf(H1)} and the strong comparison principle hold. Let $v\colon\overline Q_T\to \R$ be a viscosity solution to \eqref{eq:HJB} in the weak sense. Then it holds that 
$$
v(t,x)=V(t,x)\quad\text{for all }(t,x)\in[0,T]\times\Omega,
$$
where $V$ is the value function defined in~\eqref{value_function_exit_from_open}.
\end{prop}
\begin{proof}
Since $v$ and $V$ solve \eqref{eq:HJB} in the weak sense, Definition~\ref{weak_boundary_conditions}{\rm(iii)} and the comparison principle imply that $v^*\leq V_{*}\leq V^{*}\leq v_{*}\leq v^*$ in $[0,T]\times\Omega$. The result follows. 
\end{proof}

\section{Revisiting the Milstein-Tretyakov scheme in the linear case} 
\label{sec:milstein}
In this section, we first recall the standard fully discrete SL scheme to solve a linear version of~\eqref{parabolic_HJB} when $Q_{T}$ is replaced by $[0,T)\times\R^d$ (see, e.g.,~\cite{CamFal95,DebrJako13}). Next, we recall the semi-discrete scheme proposed in~\cite{MR1867423} to solve~\eqref{eq:HJB} in the one-dimensional and linear case and we provide an extension to the
$d$-dimensional, but still linear, case. The corresponding fully discrete scheme, described in Section~\ref{fully_discrete_linear_case_section}, will be the basis of the one that we will propose in Section~\ref{sec:fullydiscrete} to approximate solutions
to~\eqref{eq:HJB}.  
  
\subsection{The classical SL scheme in $[0,T]\times\R^d$} 
Let us consider the linear parabolic equation 
\begin{subequations}
\label{eq:linear_equaation_R_d}
\begin{align}
-\partial_{t}v+\mathcal{L}(t, x,D_xv, D^2_{x}v)&=0
\quad\text{in }[0,T)\times\R^d, \label{parabolic_HJB_unbounded}\\
v(T,\cdot)&=g(\cdot)\quad\text{in }\R^d,\label{boundary_conditions_unbounded}
\end{align}
\end{subequations}
where $\mathcal{L}\colon[0,T]\times\R^d\times\R^d\times\R^{d\times d}\to\R$ is defined by 
\begin{multline}
\label{linear_operator}
\mathcal{L}(t,x,q,P)=-\frac{1}{2}\text{Tr}\big(\sigma \sigma^\top(t,x)P\big)
-b(t,x)^{\top}q-f(t, x),\\
\text{for all }(t,x,q,P)\in [0,T]\times\R^d\times\R^d\times\R^{d\times d}, 
\end{multline}
with $b\colon[0,T]\times\R^d\to\R^d$, $\sigma\colon[0,T]\times\R^d\to\R^{d\times p}$, $f\colon[0,T]\times\R^d\to\R$, and $g\colon\R^d\to\R$ being bounded and continuous.  
 
Let us first consider a time discretization of~\eqref{eq:linear_equaation_R_d}.  Let $N\in\mathbb{N}$ be the number of time steps, set $\Delta t=T/N$, $\mathcal{I}_{\Dt}=\{0,\hdots,N\}$, 
$\mathcal{I}^{*}_{\Delta t}=\{0,\hdots,N-1\}$, and $t_k:=k\Delta t$
for $k\in \I_{\Delta t}$. Given $k\in\mathcal{I}^{*}_{\Delta t}$, $x\in \R^d$,
and $\ell\in\mathcal{I}:=\{1,\hdots,p\}$, we set 
\begin{equation}
\label{discrete_characteristics_unbounded_case}
y^{\pm,\ell}_{k}(x)=x+\Delta t b(t_k,x)\pm\sqrt{p\Dt}\sigma^{\ell}(t_k,x),
\end{equation}
where $\sigma^\ell(t_k,x)$ denotes the $\ell$-th column of the matrix $\sigma(t_k,x)$. Given a set $X$ we denote by $B(X)$ the set of bounded real-valued functions defined on $X$.  The discrete time, also called semi-discrete, SL approximation $(V_k)_{k=0}^{N} \in B(\R^d)^{N+1}$ is given by the following relation, which is solved backward in time, 
\begin{align}
\label{eq:def_Slin}
V_k(x)&=\mathcal{S}_{k}^{\text{sd}}(V_{k+1},x)\quad\text{for all }
k\in\mathcal{I}^{*}_{\Delta t},\,x\in\R^d,\nonumber\\
V_{N}(x)&=g(x)\quad\text{for all }x\in \R^d,
\end{align}
where, for every $k\in\mathcal{I}_{\Delta t}^*$,   $\mathcal{S}_{k}^{\text{sd}}\colon B(\R^d)  \times\R^d\to \R$ is defined by 
\begin{equation}
\label{forward_operator_unbounded_case}
\mathcal{S}_{k}^{\text{sd}}(\phi,x)=\frac{1}{p}\sum_{\ell=1}^{p}
\Big(\gamma_k^{+,\ell} \phi( y^{+,\ell}_k(x)) + \gamma_k^{-,\ell} \phi( y^{-,\ell}_k (x))\Big)+\tau_{k}f(t_{k},x),
\end{equation}
with 
\begin{equation}
\label{weights_unconstrained_case} 
\gamma_{k}^{+,\ell}=1/2,\quad\gamma_{k}^{-,\ell}=1/2,\quad\text{for all }
\ell=1,\hdots,p,\;\mbox{and }\tau_{k} =\Dt. 
\end{equation} 
By a second-order Taylor expansion, one can check the {\it consistency} of~\eqref{eq:def_Slin} with~\eqref{parabolic_HJB_unbounded}. More precisely, for every $\beta\in(\frac{1}{2},1]$ and $\varphi\in \mathcal{F}_{\beta}(\overline{Q}_T)$, with  compact support,  there exists  $\epsilon\colon[0,+\infty)\to [0,+\infty)$ (depending only on $\varphi$) such that
$\epsilon(\Dt)\to 0$, as $\Delta t\to 0$, and, for every $k\in\mathcal{I}^{*}_{\Dt}$, and $x\in \R^d$, one has
\begin{equation}\label{consistency_linear_unbounded_case}
\left| \frac{\varphi(t_{k},x) -  \mathcal{S}_{k}^{\text{sd}}(\varphi(t_{k+1},\cdot),x)}{\tau_k} - \left(-\partial_{t} \varphi(t_k,x) +\mathcal{L}(t_k, x, D_x \varphi(t_k,x) , D_{x}^2 \varphi(t_k,x))\right) \right|\leq \epsilon(\Dt).
\end{equation} 
\subsection{A variation of the Milstein-Tretyakov scheme in $Q_T$} Let us begin by recalling the scheme proposed in  \cite{MR1867423} for the linear equation
\begin{align}
\label{eq:linear}
-\partial_{t}v+\mathcal{L}(t,x,v_{x},v_{xx})&=0\quad\text{in }Q_T,
\\
v&=\Psi\quad\text{on }\partial^{*}Q_T,
\end{align}
in the one-dimensional case. More precisely, let us take $d=p=1$,
$Q_{T}=[0,T)\times\mathfrak{I}$, with $\mathfrak{I}\subset\R$ being an open interval, $\mathcal{L}$ given by~\eqref{linear_operator},  and $\Psi\colon\partial^{*}Q_T\to\R$ continuous.
In this one-dimensional framework, for simplicity and following the standard notation, we have denoted $\phi_{x}(t,x)=D_{x}\phi(t,x)$ and $\phi_{xx}(t,x)=D^{2}_{x}\phi(t,x)$.  

In order to obtain a modification of the semi-discrete scheme~\eqref{eq:def_Slin}, incorporating the boundary data $\Psi$, and such that a consistency estimate, similar 
to~\eqref{consistency_linear_unbounded_case}, holds,  Milstein and Tretyakov have proposed in~\cite{MR1867423} to truncate the discrete characteristics $y^{\pm}$, defined
in~\eqref{discrete_characteristics_unbounded_case}, and to modify the parameters 
$\gamma^+$, $\gamma^-$, and $\tau$ in~\eqref{forward_operator_unbounded_case}-\eqref{weights_unconstrained_case} (here we drop the index $\ell$ having $p=1$). More precisely, for $k\in\mathcal I_{\Delta t}^{*}$ and $x\in\mathfrak{I}$, let us define 
\begin{equation}
\label{definicion_lambda_caso_lineal}
\lambda^{\pm}_k(x)=\min\Big\{\inf\{\lambda\in [0,1]\,|\,x+\lambda\Delta t 
b(t_k,x)\pm\sqrt{\lambda\Dt}\sigma(t_k,x)\notin \mathfrak{I}\},1\Big\}.
\end{equation}
Notice that, by definition, $\lambda^{\pm}_k(x)\in(0,1]$. Let us also set  
\begin{equation}
\label{definizione_y_k_pm}
y^{\pm}_k(x)=x+\lambda^{\pm}_{k}(x)\Delta t b(t_{k},x)\pm\sqrt{\lambda^{\pm}_{k}(x)\Delta t}\sigma(t_k,x).
\end{equation}
 We consider the semi-discrete scheme
\begin{align}
\label{linear_scheme_bounded_case}
V_{k}(x)&=\mathcal{S}_{k}^{\text{sd}}(V_{k+1},x)\quad\text{for all }
k\in\mathcal{I}^{*}_{\Delta t},\,x\in\mathfrak{I},\nonumber\\
V_{k}(x)&=\Psi(t_k,x)\quad\text{for all }k\in\mathcal{I}^{*}_{\Dt},\,x\in\partial \mathfrak{I},\nonumber\\
V_{N}(x)&=\Psi(T,x)\quad\text{for all }x\in \overline{\mathfrak{I}},
\end{align}
where, for every $\phi\in\mathcal{B}(\overline{\mathfrak{I}})$, $k\in\mathcal{I}^{*}_{\Delta t}$, and
$x\in\mathfrak{I}$, $\mathcal{S}_{k}^{\text{sd}}(\phi,x)$ is redefined as
\begin{equation}
\label{Sh_coefficient_determination}
\mathcal{S}_{k}^{\text{sd}}(\phi,x)=\gamma^{+}_{k}(x)\tilde{\phi}^{+}_{k}(x)
+\gamma^{-}_{k}(x)\tilde{\phi}^{-}_{k}(x)+\tau_{k}(x)f(t_{k},x),
\end{equation}
with 
\begin{equation}
\label{definition_hat_linear_case}
\tilde{\phi}^{\pm}_{k}(x):=
\begin{cases}
\phi(y^{\pm}_{k} (x))&\text{if }\lambda^{\pm}_{k}(x)=1,\\
\Psi\big(t_{k}+\lambda^{\pm}_{k}(x)\Delta t,y^{\pm}_{k}(x)\big)&\text{otherwise},
\end{cases}
\end{equation}
and the scalars $\gamma^{\pm}_k(x)$, $\tau_{k}(x)\in (0,+\infty)$ in~\eqref{Sh_coefficient_determination} are determined by consistency considerations. For this purpose, and for the sake of simplicity, let us assume that $b\equiv 0$, $\sigma$ constant, and take $\varphi\in C^{\infty}(\overline{Q}_{T})$ such that $\varphi=\Psi$ in
$[0,T]\times\partial\mathfrak{I}$. For all $(k,x)\in\mathcal{I}^*_{\Delta t}\times\mathfrak{I}$, a second-order Taylor expansion yields 
\begin{align}
\label{sviluppo_dim_1_caso_b_zero}
\varphi(t_{k}+\lambda^{\pm}_{k}(x)\Delta t,y^{\pm}_{k}(x))&=
\varphi(t_{k},x)+\lambda^{\pm}_k(x)\Delta t\partial_{t}\varphi(t_k,x)\pm \sqrt{\lambda_{k}^{\pm}(x)\Delta t}\varphi_{x}(t_k,x)\sigma\nonumber\\
&\hspace{0.3cm}+\frac{\lambda_{k}^{\pm}(x)\Delta t}{2}\varphi_{xx}(t_k,x)\sigma^2
+\lambda_{k}^{\pm}(x)\Delta t\epsilon_{k}^{\pm}(\Delta t,x), 
\end{align}
where $|\epsilon_{k}^{\pm}(\Delta t,x)|\leq C_{\varphi}\sqrt{\Delta t}$, with
$C_{\varphi}\in (0,+\infty)$ being independent of $(k,x)\in\mathcal{I}^{*}_{\Delta t}
\times \mathfrak{I}$. It follows that 
\begin{multline}
\label{sums_gamma_varphi_one_dimensional_case}
\gamma^{+}_{k}(x)\varphi(t_{k}+\lambda^{+}_{k}(x)\Delta t,y^{+}_{k}(x))
+\gamma^{-}_{k}(x)\varphi(t_{k}+\lambda^{-}_{k}(x)\Delta t,y^{-}_{k}(x))\\
=(\gamma_{k}^{+}(x)+\gamma_{k}^{-}(x))\varphi(t_k,x)+
\Big(\gamma_{k}^{+}(x)\sqrt{\lambda_{k}^{+}(x)\Delta t}
-\gamma_{k}^{-}(x)\sqrt{\lambda_{k}^{-}(x)\Delta t}\Big)\varphi_{x}(t_k,x)\sigma\\
+\big(\gamma_{k}^{+}(x)\lambda_{k}^{+}(x)\Delta t
+\gamma_{k}^{-}(x)\lambda_{k}^{-}(x)\Delta t\big)\Big(\partial_{t}\varphi(t_{k},x)+\frac{\sigma^2}{2}\varphi_{xx}(t_{k},x)\Big)\\
+\gamma^{+}_{k}(x)\lambda^{+}_{k}(x)\Delta t\epsilon_{k}^{+}(\Delta t,x)
+\gamma^{-}_{k}(x)\lambda^{-}_{k}(x)\Delta t \epsilon_k^{-}(\Delta t,x),
\end{multline}
with 
\begin{equation} 
\label{errors_one_dimensional_case}
\lim_{\Delta t\to 0}\sup_{k\in\mathcal{I}_{\Delta t}^{*}, x\in\mathfrak{I}}
\Big\{|\epsilon_{k}^{+}(\Delta t,x)|\vee |\epsilon_{k}^{-}(\Delta t,x)|\Big\}=0. 
\end{equation}
Thus, by setting 
\begin{equation}
\label{weights_uni_dimensional_case}
\gamma^{\pm}_{k}(x)=\frac{\sqrt{\lambda^{\mp}_{k}(x)}}{\sqrt{\lambda^{+}_{k}(x)}
+\sqrt{\lambda^{-}_{k}(x)}}\quad\text{and}\quad\tau_{k}(x)=
\gamma_{k}^{+}(x)\lambda^{+}_{k}(x)\Delta t+\gamma^{-}_{k}(x)\lambda^{-}_{k}(x)\Delta t,
\end{equation}
we get
\begin{multline}
\gamma^{+}_{k}(x)\varphi\big(t_{k}+\lambda^{+}_{k}(x)\Delta t, y^{+}_{k}(x)\big) 
+\gamma^{-}_{k}(x)\varphi\big(t_{k}+\lambda^{-}_{k}(x)\Delta t,y^{-}_{k}(x)\big)\\
=\varphi(t_{k},x)+\tau_{k}(x)\bigg(\partial_{t}\varphi(t_{k},x) 
+\frac{\sigma^{2}}{2}\varphi_{xx}(t_{k},x)\bigg)+\gamma^{+}_{k}(x)\lambda^{+}_{k}(x)\Delta t\epsilon_{k}^{+}(\Delta t,x)+\gamma^{-}_{k}(x)\lambda^{-}_{k}(x)\Delta t\epsilon_{k}^{-}(\Delta t,x),
\end{multline}
and, in view of~\eqref{errors_one_dimensional_case}, we get the following consistency result
\begin{multline*}
\lim_{\Delta t\to 0}\sup_{k\in\I^{*}_{\Delta t},x\in\mathfrak{I}}\Bigg|\frac{\varphi(t_{k},x)- \mathcal{S}^{\rm sd}_{k}(\varphi(t_{k+1},\cdot),x)}{\tau_{k}(x)}
-\big(-\partial_{t}\varphi(t_{k},x)+\mathcal{L}(t_{k},x,\varphi_{x}(t_{k},x),
\varphi_{xx}(t_{k},x))\big)\Bigg|=0.
\end{multline*}

Now, let us discuss an extension of the previous semi-discrete scheme to the multi-dimensional case, i.e. $Q_T=[0,T)\times\Omega$, with $\Omega\subset\R^d$ ($d\geq 2$) being a nonempty and bounded domain. Let $k\in\mathcal{I}^{*}_{\Dt}$, $x\in\Omega$,
$\ell\in\mathcal{I}$, and define 
\begin{equation}
\label{lambda}
\lambda^{\pm,\ell}_{k}(x)=\min\Big\{\inf\{\lambda\in [0,1]\,\big|\,x+\lambda\Delta t b(t_{k},x)\pm\sqrt{p\lambda\Delta t}\sigma^{\ell}(t_{k},x)\notin\Omega\},1\Big\}.
\end{equation}
As in the one-dimensional case, we have $\lambda^{\pm,\ell}_{k}(x)\in(0,1]$. Let us set
\begin{equation}
\label{eq:yell}
y^{\pm,\ell}_{k}(x)=x+\lambda^{\pm}_{k}(x)\Delta t b(t_{k},x)\pm
\sqrt{p\lambda^{\pm}_{k}(x)\Delta t}\sigma^{\ell}(t_{k},x).
\end{equation}
The choices~\eqref{weights_uni_dimensional_case}, in the one-dimensional case, suggest to define
\begin{multline}
\label{weights_multi_dimensional_case}
\gamma^{\pm,\ell}_{k}(x)=\frac{\sqrt{\lambda^{\mp,\ell}_{k}(x)}}{\sqrt{\lambda^{+,\ell}_{k}(x)}+\sqrt{\lambda^{-,\ell}_{k}(x)}}\\
\mbox{and}\quad\tau^{\ell}_{k}(x)=\gamma_{k}^{+,\ell}(x)\lambda^{+,\ell}_{k}(x)\Delta t
+\gamma^{-,\ell}_{k}(x)\lambda^{-,\ell}_{k}(x)\Delta t=\Delta t\sqrt{\lambda^{+,\ell}_{k}(x)\lambda^{-,\ell}_{k}(x)},
\end{multline}
which is supported by the following extension of~\eqref{sums_gamma_varphi_one_dimensional_case}, whose technical proof is deferred to the Appendix at the end of the paper. 
\begin{lem}  
\label{Taylor_dim_1}
Let $\beta\in(\frac{1}{2},1]$ and let $\varphi\in\mathcal{F}_{\beta}(\overline{Q}_T)$. Then, for every $k\in\I_{\Delta t}^{*}$, $x\in\Omega$, and $\ell\in\mathcal{I}$, the following holds: 
\begin{multline}
\label{somma_taylor_pm}
\gamma^{+,\ell}_k(x)\varphi\Big(t_{k}+\lambda^{+,\ell}_{k}(x)\Delta t, y^{+,\ell}_{k}(x)\Big) 
+\gamma^{-,\ell}_{k}(x)\varphi\Big(t_{k}+\lambda^{-,\ell}_{k}(x)\Delta t,y^{-,\ell}_{k}(x)\Big)\\ 
=\varphi(t_{k},x)+\tau^{\ell}_{k}(x)\bigg(\partial_{t}\varphi(t_{k},x) + D_{x} \varphi(t_{k},x)^{\top} b(t_{k},x)+\frac{p}{2}\sigma^{\ell}(t_{k},x)^{\top}
\big[D_{x}^{2}\varphi(t_{k},x)\sigma^{\ell}(t_{k},x)\big]\bigg)\\ 
+\gamma^{+,\ell}_{k}(x)\lambda^{+,\ell}_{k}(x)\Delta t\epsilon_{k}^{+,\ell}(\Dt,x)+\gamma^{-,\ell}_{k}(x)\lambda^{-,\ell}_{k}(x)\Delta t\epsilon_{k}^{-,\ell}(\Delta t,x),
\end{multline}
where $\epsilon_{k}^{\pm,\ell}\colon [0,+\infty)\times \Omega\to\R$ satisfies  
\begin{equation}
\label{errore_globale_apendice}
\lim_{\Delta t\to 0}\sup_{k\in\mathcal{I}_{\Delta t}^{*},\,x\in\Omega} \Big\{|\epsilon_{k}^{+,\ell}(\Delta t,x)|\vee |\epsilon_{k}^{-,\ell}(\Delta t,x)|\Big\}= 0. 
\end{equation}
\end{lem}

Given  $k\in\I_{\Delta t}^{*}$ and $x\in\Omega$, define  
\begin{equation}
\label{eq:tau}
\pi^{\ell}_{k}(x)=
\begin{cases} 
1 & \text{if }p=1,\\
\frac{\prod_{\ell_{1} \neq \ell}\tau^{\ell_{1}}_{k}(x)}{\sum_{\ell_{2}=1}^{p}
\prod_{\ell_{3}\neq\ell_{2}}\tau^{\ell_{3}}_{k}(x)} &\text{otherwise},
\end{cases}
\quad\text{for all }\ell\in\mathcal{I},  
\end{equation}
and set
\begin{equation}
\label{total_tau}
\tau_{k}(x)= 
\begin{cases}
\tau_{k}^{1}(x) &\text{if }p=1,\\
p\frac{\prod_{\ell=1}^{p}\tau^{\ell}_{k}(x)}{\sum_{\ell_{2}=1}^{p}
\prod_{\ell_{3}\neq\ell_{2}}\tau^{\ell_{3}}_{k}(x)} &\text{otherwise}.
\end{cases}
\end{equation}
Observe that the following identity holds:
\begin{equation}
\label{tau_ell_pi_ell}
\sum_{\ell\in\mathcal{I}}\pi^{\ell}_{k}(x)=1\quad\text{and}\quad\pi^{\ell}_{k}(x)\tau^{\ell}_{k}(x)=\frac{\tau_{k}(x)}{p}\quad\text{for all }\ell\in\mathcal{I}.
\end{equation}
For all $\phi\in\mathcal{B}(\overline{\Omega})$, $k\in\I_{\Delta t}^{*}$, and
$x\in\Omega$, define 
\begin{equation}
\label{definition_hat_linear_multidimensional_case}
\tilde{\phi}^{\pm,\ell}_{k}(x)= 
\begin{cases}
\phi (y^{\pm,\ell}_{k}(x))&\text{if }\lambda^{\pm,\ell}_{k}(x)=1,\\
\Psi\Big(t_{k}+\lambda^{\pm,\ell}_{k}(x)\Delta t,y^{\pm,\ell}_{k}(x)\Big)& \text{otherwise},
\end{cases}
\quad\text{for all }\ell\in\mathcal{I},  
\end{equation}
and set
\begin{equation}
\label{definition_semi_discrete_scheme_S_bounded_case}
\mathcal{S}_{k}^{\text{sd}}(\phi,x)=\sum_{\ell=1}^{p}\pi^{\ell}_{k}(x) 
\Big(\gamma^{+,\ell}_{k}(x)\tilde{\phi}^{+,\ell}_{k}(x)+\gamma^{-,\ell}_{k}(x)
\tilde{\phi}^{-,\ell}_{k}(x)\Big)+\tau_{k}(x)f(t_{k},x).
\end{equation}

We consider the following semi-discrete scheme:
\begin{align}
\label{linear_scheme_bounded_multidimensional_case}
V_{k}(x)&=\mathcal{S}_{k}^{\text{sd}}(V_{k+1},x)\quad\text{for all }
k\in\mathcal{I}^{*}_{\Delta t},\,x\in\Omega,\nonumber\\
V_{k}(x)&= \Psi(t_{k},x)\quad\text{for all }k\in\mathcal{I}^{*}_{\Delta t},
\,x\in\partial\Omega,\nonumber\\
V_{N}(x)&=\Psi(T,x)\quad\text{for all }x\in\overline{\Omega}.
\end{align}
\begin{rem}
\label{rem:consistency_with_unbounded_case}
\begin{enumerate}[{\rm(i)}] 
\item 
\label{rem:consistency_with_unbounded_case_i}
Let $k\in\mathcal{I}^*_{\Dt}$, $x\in\Omega$, and suppose that, for every $\ell\in\mathcal{I}$, the discrete characteristics $y^{\pm,\ell}_{k}(x)$ do not exit the domain, i.e. $\lambda^{\pm,\ell}_{k}(x)=1$. Then~\eqref{definition_semi_discrete_scheme_S_bounded_case},~\eqref{eq:tau},~\eqref{weights_multi_dimensional_case}, and~\eqref{definition_hat_linear_multidimensional_case} yield 
\begin{equation*}
\mathcal{S}_{k}^{\rm{sd}}(\phi,x)=\frac{1}{2p}\sum_{\ell=1}^{p}
\Big(\phi(y^{+,\ell}_k(x))+\phi(y^{-,\ell}_{k}(x))\Big)+\Delta t f(t_{k},x),
\end{equation*}
which coincides with \eqref{forward_operator_unbounded_case}.

\item
\label{rem:consistency_with_unbounded_case_ii}
Scheme~\eqref{linear_scheme_bounded_multidimensional_case} differs from the one proposed in~\cite[Section 6]{MR1867423} in that the discrete 
characteristics~\eqref{definizione_y_k_pm}, starting at $x$ at time $t_k$, are constructed by considering a specific column $\sigma^\ell(t_k,x)$ and not the entire matrix
$\sigma(t_{k},x)$. In particular, in view of~\eqref{rem:consistency_with_unbounded_case_i}, definition~\eqref{definition_semi_discrete_scheme_S_bounded_case} is the natural modification of~\eqref{forward_operator_unbounded_case} to deal with Dirichlet boundary conditions. In this framework, the proof of Proposition~\ref{consistency_semi_discrete_scheme_linear_case} below shows that the expressions for the weights $\pi_{k}^{\ell}$ and the rescaled time $\tau_{k}(x)$ in~\eqref{eq:tau} are the only possible choices which provide a consistent scheme.
\end{enumerate}
\end{rem}

\begin{prop}
\label{consistency_semi_discrete_scheme_linear_case}
Let $\beta\in (\frac{1}{2},1]$ and let $\varphi\in\mathcal{F}_{\beta}(\overline{Q}_T)$ be such that $\varphi=\Psi$ on $[0,T]\times\partial\Omega$. Then the following consistency result holds: as $\Delta t\to 0$, we have
\begin{equation}
\label{consistency_linear_bounded_semidiscrete_case}
\sup_{k\in\mathcal{I}^{*}_{\Delta t},\,x\in\Omega}\Bigg| \frac{\varphi(t_{k},x)-\mathcal{S}_{k}^{\rm{sd}}(\varphi(t_{k+1},\cdot),x)}{\tau_{k}(x)} -\Big(-\partial_{t}\varphi(t_{k},x)+\mathcal{L}(t_{k},x,D_{x}\varphi(t_{k},x),D_{x}^2 \varphi(t_{k},x))\Big)\Bigg|\to 0.
\end{equation}  
\end{prop}
\begin{proof} 
Let $k\in\mathcal{I}^*_{\Delta t}$ and $x\in\Omega$.
Using~\eqref{definition_semi_discrete_scheme_S_bounded_case},~\eqref{tau_ell_pi_ell},
~\eqref{somma_taylor_pm} in Lemma~\ref{Taylor_dim_1}, and the identity  
\begin{equation}
\sum^{p}_{\ell=1}\sigma^{\ell}(t_{k},x)^{\top} [D_{x}^2\varphi(t_{k},x)\sigma^{\ell}(t_{k},x)]=\text{Tr}[\sigma\sigma^{\top}(t,x)D_{x}^{2}\varphi(t_{k},x_{i})],
\end{equation}
we have that
\begin{multline}
\Big|\varphi(t_{k},x)-\mathcal{S}_{k}^{\text{sd}}(\varphi(t_{k+1},\cdot),x)
+\tau_{k}(x)\big(\partial_{t}\varphi(t_{k},x)-\mathcal{L}(t_{k},x,D_{x}\varphi(t_{k},x), D_{x}^2\varphi(t_{k},x))\big)\Big|\\
\leq\tau_{k}(x)\max_{\ell\in\mathcal{I}}\sup_{k\in\mathcal{I}^{*}_{\Dt},\,x\in\Omega}
\Big\{\big|\epsilon_{k}^{+,\ell}(\Delta t,x)\big|\vee\big|\epsilon_{k}^{-,\ell}(\Delta t,x)\big|\Big\}.
\end{multline}
Thus, dividing the previous inequality by $\tau_{k}(x)$, relation  ~\eqref{consistency_linear_bounded_semidiscrete_case} follows from~\eqref{errore_globale_apendice}.
\end{proof}
 
\subsection{A fully discrete linear SL scheme  in $Q_T$}
\label{fully_discrete_linear_case_section} 
We describe now a fully discrete SL scheme to solve
\begin{subequations}
\begin{align}
-\partial_{t}v(t,x)+\mathcal{L}(t,x,D_{x}v(t,x),D^{2}_{x}v(t,x))&=0
\quad\text{in }Q_T,
\label{parabolic_HJB_bounded}\\
v(t,x)&=\Psi(x)\quad\text{in }\partial^{*}Q_{T},
\label{boundary_conditions_bounded}
\end{align}
\end{subequations}
where $\mathcal{L}$ is the linear operator defined in \eqref{linear_operator}.

Consider a triangulation $\widehat{\mathsf{T}}_{\Delta x}$ of $\Omega$ with maximum mesh size $\Delta x$ which is exact, i.e. the boundary elements of $\widehat {\mathsf{T}}_{\Delta x}$ are possibly curved and $\overline{\Omega }= \bigcup_{ \widehat{\mathsf{T}} \in \widehat{\mathsf{T}}_{\Delta x}} \widehat{\mathsf{T}}$ (see, e.g.,~\cite{B89}).
We denote by $\mathcal{G}_{\Delta x} = \left\{x_{i}\,|\,i\in\I_{\Delta x}\right\}$, with $\I_{\Delta x}=\{1,\hdots, N_{\Delta x}\}$, the set of vertices of the triangulation. Following~\cite{Deckelnick09,CCDS21}, we also consider an associated finite subdivision $ {\mathsf{T}}_{\Delta x}$ consisting of simplicial finite elements ${\mathsf{T}}$, with vertices in $\mathcal{G}_{\Delta x}$, and denote by $\Omega_{\Delta x}=\bigcup_{\mathsf{T}\in \mathsf{T}_{\Delta x}}\mathsf{T}$ the resulting closed polyhedral domain.  We suppose that every $\mathsf{T}\in\mathsf{T}_{\Delta x}$ shares the same set of vertices than an element $\widehat{\mathsf{T}}\in\widehat{\mathsf{T}}_{\Delta x}$ and that 
$\mathsf{T}_{\Delta x}$ is a regular triangulation, i.e. there exists $\delta\in(0,1)$, independent of $\Delta x$, such that  each $\mathsf{T}\in\mathsf{T}_{\Delta x}$ is contained in a 
ball of radius $\Delta x/\delta$ and $\mathsf{T}$ contains a ball of radius $\delta\Delta x$. For every  $\mathsf{T}\in \mathsf{T}_{\Delta x}$ and $x\in\overline{\Omega}$, denote by $p_{\mathsf{T}}(x)$ the projection of $x$ onto $\mathsf{T}$ and define 
$p_{\Delta x}\colon\overline{\Omega}\rightarrow
\Omega_{\Delta x}$ by
\begin{equation}
p_{\Delta x}(x)=p_{\mathsf{T}}(x) \quad 
\mbox{if $x\in \widehat{\mathsf{T}}\in \widehat {\mathsf{T}}_{\Delta x}$  and the element $\mathsf{T}\in {\mathsf{T}}_{\Delta x}$ has the same vertices than  $\widehat{ \mathsf{T}} $}.
\end{equation}
In what follows, we assume the existence of  $c>0$, independent of $\Delta x$, such that, for $\Delta x$ small enough, 
\begin{equation}
\sup_{x\in\overline{\Omega}}|x-p_{\Delta x}(x)|\leq c(\Delta x)^{2}.
\label{hyp:distance_triangles}
\end{equation}
Constructions of curved and polyhedral triangulations of two and three dimensional domains, with $C^2$ boundaries, such that~\eqref{hyp:distance_triangles} holds can be found, for instance, in~\cite[Section~4]{DziukElliott2013}.

Let $\{\psi_{i}\,| \,i\in \mathcal{I}_{\Delta x}\}$  be the linear finite element basis $\mathbb{P}_1$ on $\mathsf{T}_{\Delta x}$ and, for every
$\phi\colon\mathcal{G}_{\Delta x}\to\R$, denote by $I[\phi]$ its linear interpolation on $\widehat{\mathsf{T}}_{\Delta x}$, defined by
\begin{equation}
\label{interpolation} 
I[\phi](x):=\sum_{i=1}^{N_{\Delta x}}\psi_{i}({p_{\Dx}}(x))\phi_{i}\quad\text{for all } x\in\overline{\Omega},
\end{equation}
where, for notational simplicity, hereafter we set $\phi_{i}=\phi(x_{i})$. Using standard interpolation results and the regularity of the mesh, for every $\varphi\in C^2(\R^d)$, there exists $c_{\varphi}>0$, independent of $\Delta x$, such that 
\begin{equation}
\label{interpolation_ineq} 
\sup_{x\in\Omega_{\Delta x}}\; \left|\varphi(x)-\sum_{i=1}^{N_{\Delta x}}\psi_{i}(x)\varphi(x_{i})\right|\leq c_{\varphi} (\Delta x)^2\quad\text{for all }x\in\overline{\Omega},
\end{equation}
see e.g.,~\cite[Theorem 16.1]{Ciarlet}.  Thus, setting $C_{\varphi}=c+c_{\varphi}$, it follows from~\eqref{hyp:distance_triangles},~\eqref{interpolation_ineq}, and the triangular inequality that
\be
\label{interp_estim}
\sup_{x\in\overline{\Omega}}\; \big|\varphi(x) - I \left[ \varphi|_{\mathcal{G}_{\Delta x}}\right] (x)\big|\leq C_{\varphi} (\Delta x)^2.
\ee 
Notice that~\eqref{interp_estim} also holds for every $\varphi\in C^{2}(\overline{\Omega})$ as one can extend $\varphi$ to an element in $C^{2}(\R^{d})$ (see, e.g.,~\cite[Lemma 2.20]{Lee2012}). Let $\partial\mathcal{I}_{\Delta x}=\{i\in\mathcal I_{\Delta x}\,|\,
x_{i}\in\partial\Omega\}$, ${\mathcal{I}}^{\circ}_{\Delta x}=
\mathcal{I}_{\Delta x}\setminus\partial\mathcal{I}_{\Delta x}$ and, for every $k\in\I_{\Delta}^{*}$, $i\in {\mathcal{I}}^{\circ}_{\Delta x}$, and $\ell\in\I$, define $ y^{\pm,\ell}_{k,i}=y^{\pm,\ell}_k(x_{i})$ by~\eqref{eq:yell}. Recalling~\eqref{lambda},~\eqref{weights_multi_dimensional_case},
~\eqref{eq:tau},~\eqref{total_tau}, let us set
\begin{equation}
\label{def:notation_under}
\lambda^{\pm,\ell}_{k,i}=\lambda^{\pm,\ell}_k(x_{i}),\quad \gamma^{\pm,\ell}_{k,i}=\gamma^{\pm,\ell}_{k}(x_{i}),\quad \tau^{\ell}_{k,i}=\tau_{k}^{\ell}(x_{i}),\quad
\pi^\ell_{k,i}=\pi^{\ell}_{k}(x_{i}),\quad\tau_{k,i}=\tau_{k}(x_{i}).
\end{equation}

We consider the fully discrete SL approximation $\{V_{k}\colon\mathcal{G}_{\Delta x}\to\R\,|\,k\in\mathcal{I}_{\Delta t}\}$, defined by the following backward recursion:
\begin{align}
\label{linear_fully}
V_{k,i}&=\mathcal{S}_{k,i}^{\text{fd}}(V_{k+1})\quad\text{for all }
k\in\mathcal{I}^{*}_{\Delta t},\,i\in {\mathcal{I}}^{\circ}_{\Delta x},\nonumber\\ 
V_{k,i}&=\Psi(t_{k},x_{i})\quad\text{for all }k\in\mathcal{I}^{*}_{\Delta t},\,
i\in\partial\mathcal{I}_{\Delta x},\nonumber\\
V_{N,i}&=\Psi(T,x_{i})\quad\text{for all }i\in\mathcal{I}_{\Delta x},
\end{align}
where, for every $\phi\colon\mathcal{G}_{\Delta x}\to\R$, $k\in\mathcal{I}^{*}_{\Delta t}$, and $i\in {\mathcal{I}}^{\circ}_{\Delta x}$,   
\begin{equation}
\label{definition_fully_discrete_scheme_S_bounded_case}
\mathcal{S}_{k,i}^{\text{fd}}(\phi)=\sum_{\ell=1}^{p}\pi^{\ell}_{k,i}
\Big(\gamma^{+,\ell}_{k,i}\tilde{\phi}^{+,\ell}_{k,i}+\gamma^{-,\ell}_{k,i}
\tilde{\phi}^{-,\ell}_{k,i}\Big)+\tau_{k,i}f(t_{k},x_i), 
\end{equation}
with $\tilde{\phi}^{\pm,\ell}_{k,i}$ $(\ell\in\I)$ being given by
\begin{equation}
\label{definition_hat_linear_multidimensional_case_fully_discrete}
 \tilde{\phi}^{\pm,\ell}_{k,i} := \begin{cases}
   I[\phi]\big( y^{\pm,\ell}_{k,i} \big)   &  \text{if } \lambda^{\pm,\ell}_{k,i}=1,\\
 \Psi\left(t_k+\lambda^{\pm,\ell}_{k,i}\Dt, y^{\pm,\ell}_{k,i}\right)& \text{otherwise}.
 \end{cases}
\end{equation}

For every $\phi\colon\overline{Q}_{T}\to\R$ and $k\in\mathcal{I}_{\Delta t}$, let us set $\phi_{k}=\phi(t_k,\cdot)|_{\mathcal{G}_{\Delta x}}$. We have the following result concerning the consistency of the scheme \eqref{linear_fully}.

\begin{prop}
\label{consistency_fully_linear_case}
Let $\beta\in (\frac{1}{2},1]$ and let $\varphi\in\mathcal{F}_{\beta}(\overline{Q}_T)$ be such that $\varphi=\Psi$ on $\partial^{*}Q_{T}$. Then the following consistency property holds: as $\Delta t\to 0$ and $(\Delta x)^2/\Delta t\to 0$, we have
\begin{equation}
\sup_{k\in\mathcal{I}^{*}_{\Delta t},\,i\in {\mathcal{I}}^{\circ}_{\Delta x}}\Bigg| \frac{\varphi(t_{k},x_{i})-\mathcal{S}_{k,i}^{\text{{\rm fd}}}(\varphi_{k+1})}{\tau_{k,i}} -\Big(-\partial_{t}\varphi(t_{k},x_{i})+\mathcal{L}(t_{k},x_{i},D_{x}\varphi(t_{k},x_{i}),D_{x}^2\varphi(t_{k},x_{i}))\Big)\Bigg|\to 0.
\end{equation}  
\end{prop}
\begin{proof}  Let us fix $k\in \I^{*}_{\Delta t}$ and $i\in  {\I}^{\circ}_{\Delta x}$. It follows from \eqref{definition_hat_linear_multidimensional_case_fully_discrete} that
$$
\begin{array}{l}
\gamma^{+,\ell}_{k,i}\tilde{\varphi}^{+,\ell}_{k,i} +  \gamma^{-,\ell}_{k,i}\tilde \varphi^{-,\ell}_{k,i}
\\[10pt]
=  \gamma^{+,\ell}_{k,i} \varphi \left(t_k+\lambda^{+,\ell}_{k,i}\Dt, y^{+,\ell}_{k,i} \right) +  \gamma^{-,\ell}_{k,i}  \varphi \left(t_k+\lambda^{-,\ell}_{k,i}\Dt, y^{-,\ell}_{k,i}\right) +\gamma^{+,\ell}_{k,i}\eta^{+,\ell}_{k,i} +\gamma^{-,\ell}_{k,i}\eta^{-,\ell}_{k,i},
\end{array}
$$
for all  $\ell\in \mathcal I$, where
\begin{equation}
\label{eta_k_pm}
\eta^{\pm,\ell}_{k,i} =\begin{cases}     0 & \text{if } \lambda^{\pm,\ell}_{k,i} < 1, \\
  							I[\varphi_{k+1}]\big(y^{\pm,\ell}_{k,i}\big)- \varphi \left(t_{k+1}, y^{\pm,\ell}_{k,i}\right) & \text{if } \lambda^{\pm,\ell}_{k,i} =1.
				\end{cases}
\end{equation}
It follows from Lemma~\ref{Taylor_dim_1} that 
\begin{align}
& \gamma^{+,\ell}_{k,i}\tilde{\varphi}^{+,\ell}_{k,i}+  \gamma^{-,\ell}_{k,i}\tilde \varphi^{-,\ell}_{k,i} \nonumber
\\
& = \varphi(t_k,x_i) + \tau^{\ell}_{k,i} \Big( \partial_t \varphi(t_k,x_i) + D_{x} \varphi(t_k,x_i)^{\top} b(t_k,x_i) + \frac{p}{2}\sigma^{\ell}(t_k,x_i)^{\top} \left[ D_{x}^{2}\varphi(t_k,x_i)\sigma^{\ell}(t_k,x_i)\right]\Big) \label{eq:addendi} \\
& \quad + \gamma^{+,\ell}_{k,i}\left(\lambda^{+,\ell}_{k,i} \Dt\epsilon_{k}^{+,\ell}(\Dt,x_{i}) + \eta^{+,\ell}_{k,i} \right) +\gamma^{-,\ell}_{k,i}\left(\lambda^{-,\ell}_{k,i} \Dt\epsilon_{k}^{-,\ell}{(\Dt,x_{i})} +  \eta^{-,\ell}_{k,i}\right),\nonumber
\end{align}
where the functions $\{\epsilon_{m}^{\pm,\ell}(\cdot,x_{j})\colon [0,+\infty)\to \R \, | \, m\in \I^*_{\Delta t}, \, j\in  {\I}^{\circ}_{\Delta x}\}$ satisfy 
\begin{equation}
\label{epsilon_h_tende_0}
\epsilon(\Dt):= \max\{ |\epsilon_{m}^{\pm,\ell}(\Dt,x_{j})| \, | \, m\in \I^*_{\Delta t}, \, j\in\overset{\circ}{\I}_{\Delta x}, \, \ell\in \mathcal I\} \underset{\Dt \to 0}{\longrightarrow} 0. 
\end{equation}
Recalling~\eqref{weights_multi_dimensional_case}, we have that 
\begin{equation}
\label{resto_taylor_1}
\left| \gamma^{+,\ell}_{k,i} \lambda^{+,\ell}_{k,i} \Delta t\epsilon^{+,\ell}_{k}(\Dt,x_{i})   +\gamma^{-,\ell}_{k,i} \lambda^{-,\ell}_{k,i} \Dt \epsilon_{k}^{-,\ell}(\Dt,x_{i}) \right| \leq \tau^{\ell}_{k,i} \epsilon(\Dt) 
\end{equation}
and
\begin{equation}
\gamma^{+,\ell}_{k,i} \eta^{+,\ell}_{k,i}+\gamma^{-,\ell}_{k,i} \eta^{-,\ell}_{k,i}=  \tau_{k,i}^\ell R_{k,i}^{\ell},
\end{equation}
where
\begin{equation} 
\label{resto_taylor_2}
R_{k,i}^{\ell} :=\frac{\gamma^{+,\ell}_{k,i} \eta^{+,\ell}_{k,i}+\gamma^{-,\ell}_{k,i} \eta^{-,\ell}_{k,i}}{\Delta t\sqrt{ \lambda^{+,\ell}_{k,i}\lambda^{-,\ell}_{k,i}} }.
\end{equation}
Thus, multiplying~\eqref{eq:addendi} by  $\pi^{\ell}_{k,i}$  and taking the sum over $\ell\in \mathcal I$, we deduce from~\eqref{tau_ell_pi_ell},~\eqref{resto_taylor_1}, and~\eqref{resto_taylor_2}, that
\begin{equation*}
\sup_{k\in\mathcal{I}^{*}_{\Delta t},\,i\in\overset{\circ}{\mathcal{I}}_{\Delta x}}\Bigg| \frac{\varphi(t_{k},x_{i})-\mathcal{S}_{k,i}^{\text{{\rm fd}}}(\varphi_{k+1})}{\tau_{k,i}} -\Big(-\partial_{t}\varphi(t_{k},x_{i})+\mathcal{L}(t_{k},x_{i},D_{x}\varphi(t_{k},x_{i}),D_{x}^2\varphi(t_{k},x_{i}))\Big)\Bigg|\leq \mathcal E(\Dt, \Dx)
\end{equation*}
holds with 
\begin{equation}
\label{big_eps_prova_consistenza}
  \mathcal E(\Dt,\Dx):=   \max\left\{\left|R_{m,j}^{\ell}\right| \, \big| \, \,  m\in \I^{*}_{\Delta t}, \, j\in {\I}^{\circ}_{\Delta x}, \,  \ell\in \mathcal I\right\}+\epsilon(\Dt).
\end{equation}
Finally, let us check that $\mathcal E(\Dt,\Dx)\to 0$ if $\Dt\to 0$ and $(\Delta x)^2/\Dt \to 0$. Let us fix $m\in \I^*_{\Delta t}$, $j\in {\I}^{\circ}_{\Delta x}$, and $\ell\in \mathcal I$.  We have the following cases:
\begin{itemize}
\item[{\rm(a)}] $\lambda^{\pm,\ell}_{m,j} = 1$: Here, $\gamma_{m,j}^{\pm,\ell} = 1/2$ and, by \eqref{resto_taylor_2},  \eqref{eta_k_pm}, and \eqref{interp_estim},   we have $\left|R_{m,j}^{\ell}\right|\leq C_{\varphi} (\Delta x)^2/\Dt$. 

\item[{\rm(b)}] $\lambda^{+,\ell}_{m,j} <1$ and $\lambda^{-,\ell}_{m,j} =1$: In this case, $\eta^{+,\ell}_{m,j}=0$ and, by~\eqref{weights_multi_dimensional_case}, 
$\gamma^{-,\ell}_{m,j}/\sqrt{ \lambda^{+,\ell}_{m,j}}=1/\big(\sqrt{\lambda^{+,\ell}_{m,j}}+1\big)$. Thus, it follows from~\eqref{resto_taylor_2},~\eqref{eta_k_pm}, and~\eqref{interp_estim}, that
$$
\left|R_{m,j}^{\ell}\right|=\frac{\gamma^{-,\ell}_{m,j} \eta^{-,\ell}_{m,j}}{\Dt \sqrt{ \lambda^{+,\ell}_{m,j}}}\leq \frac{C_{\varphi} (\Delta x)^2}{\Delta t\left(\sqrt{\lambda^{+,\ell}_{m,j}}+1\right)}\leq  \frac{C_{\varphi} (\Delta x)^2}{\Dt}.
$$

\item[{\rm(c)}] $\lambda^{+,\ell}_{m,j} =1$ and $\lambda^{-,\ell}_{m,j} <1$: As in {\rm(b)}, we have  $\left|R_{m,j}^{\ell}\right|\leq C_{\varphi} (\Delta x)^2/\Dt$.
\item[{\rm(d)}] $\lambda^{\pm,\ell}_{m,j} <1$ and $\lambda^{-,\ell}_{m,j} <1$: Here, $\eta^{\pm,\ell}_{m,j}=0$ and hence, by \eqref{resto_taylor_2}, $R_{m,j}^{\ell}=0$.
\end{itemize}

Thus, the result follows from~\eqref{big_eps_prova_consistenza}, \eqref{epsilon_h_tende_0}, and cases {\rm(a)}-{\rm(d)} above. 
\end{proof}

\section{The fully discrete scheme for the HJB equation}
\label{sec:fullydiscrete}
In this section, we present the fully discrete scheme to approximate the solution to \eqref{eq:HJB}. For the time and space discretization steps, we use the the notation of the previous section. Furthermore, for every  $(k,i)\in\I^*_{\Delta t}\times {\I}^{\circ}_{\Delta x}$, $\ell\in\I$, and  $a\in A$, let us define $\lambda^{\ell, \pm}_{k,i}(a)$, $\gamma^{\pm,\ell}_{k,i}(a)$, $\tau_{k,i}^{\ell}(a)$, $\tau_{k,i}(a)$, $\pi^{\ell}_{k,i}(a)$, $y^{\pm,\ell}_{k,i}(a)$, and $\tilde{\phi}^{\pm,\ell}_{k,i}(a)$ as in the previous section, with $b(t,x)$ and $\sigma^{\ell}(t,x)$ being replaced by $b(t,x,a)$ and $\sigma^{\ell}(t,x,a)$, respectively, and, for every $\phi\colon\mathcal G_{\Delta x}\to\R$, let us set
\begin{equation}
\label{operatore_con_controllo}
\mathcal{S}^{\rm fd}_{k,i}(\phi,a)= \sum_{\ell=1}^{p}\pi^{\ell}_{k,i}(a)\left[\gamma^{+,\ell}_{k,i}(a)\tilde{\phi}^{+,\ell}_{k ,i}(a) +  \gamma^{-,\ell}_{k,i}(a)\tilde \phi^{-,\ell}_{k,i}(a)\right]  + \tau_{k,i} (a)f(t_{k},x_i,a).
\end{equation}
 
The fully discrete SL approximation $\{V_{k}\colon\mathcal G_{\Dx}\to\R\,|\,k\in \mathcal I_{\Dt}\}$ is defined by 
\begin{align}
V_{k,i}&=\inf_{a\in A}\mathcal{S}^{\rm fd}_{k,i}(V_{k+1},a)\quad\text{for all }k\in\mathcal{I}^*_{\Dt},\,i\in {\mathcal{I}}^{\circ}_{\Dx},\nonumber\\ 
V_{k,i}&=\Psi(t_k,x_i)\quad\text{for all }k\in\mathcal I^*_{\Dt},\,i\in \partial{\mathcal{I}}_{\Dx},\label{nonlinear_fully}\\
V_{N,i}&=\Psi(T,x_i)\quad\text{for all }i\in {\mathcal{I}}_{\Dx},\nonumber
\end{align}

Notice that the scheme above is explicit and admits a unique solution $\{V_{k}\colon\mathcal G_{\Dx}\to\R\,|\,k\in \mathcal I_{\Dt}\}$. Moreover, since $\Omega$ is bounded, {\bf(H1)} implies that $b$ and $\sigma$ are bounded in $\overline{Q}_{T}\times A$ and hence, arguing as in the proof of 
Proposition~\ref{consistency_fully_linear_case}, we have the following result. 
\begin{lem}
\label{lem:consistency_aux}
Assume {\bf(H1)}, let $\beta\in (\frac{1}{2},1]$, and let $\varphi\in\mathcal{F}_{\beta}(\overline{Q}_T)$ be such that $\varphi=\Psi$ on $\partial^{*}Q_{T}$. Then as $\Delta t\to 0$ and $(\Delta x)^2/\Delta t\to 0$, we have
\begin{equation}
\sup_{k\in\mathcal{I}^{*}_{\Delta t},\,i\in {\mathcal{I}}^{\circ}_{\Delta x},\,a\in A} \left| \frac{\varphi(t_{k},x_i) -\mathcal S^{\rm fd}_{k,i}(\varphi_{k+1},a)}{\tau_{k,i}(a)} - \left(-\partial_{t} \varphi(t_{k},x_{i}) + \mathcal{L}^a(t_k, x_i, D_{x}\varphi(t_{k},x_{i}) , D_{x}^2 \varphi (t_{k},x_{i}) )\right) \right|\to 0.
\end{equation}
\end{lem}

The estimate in the previous lemma directly yields the following consistency result.
\begin{prop}
\label{consistenza_caso_nonlineare}
{\em{(Consistency)}} 
Assume {\bf(H1)}, let $\beta\in (\frac{1}{2},1]$, and let $\varphi\in\mathcal{F}_{\beta}(\overline{Q}_T)$ be such that $\varphi=\Psi$ on $\partial^{*}Q_{T}$. Then the following consistency result holds: as $\Delta t\to 0$ and $(\Delta x)^2/\Delta t\to 0$, we have
\begin{multline}
\label{eq:consistency_controlled_case}
\sup_{k\in\mathcal{I}^{*}_{\Delta t},\,i\in {\mathcal{I}}^{\circ}_{\Delta x}}\Bigg|\sup_{a\in A}\Bigg(\frac{\varphi(t_{k},x_{i})-\mathcal{S}_{k,i}^{\text{{\rm fd}}}(\varphi_{k+1},a)}{\tau_{k,i}(a)}\Bigg)-\\
\Big(-\partial_{t}\varphi(t_{k},x_{i})+\sup_{a\in A}\mathcal{L}^{a}\big(t_{k},x_{i},D_{x}\varphi(t_{k},x_{i}),D_{x}^2\varphi(t_{k},x_{i})\big)\Big)\Bigg|\to 0.
\end{multline}  
\end{prop}

In order to prove the convergence of solutions of the numerical scheme  towards the viscosity solution to~\eqref{eq:HJB},  we will also need the following properties of the scheme. 
\begin{prop}
\label{prop:properties}
The following hold:
\begin{enumerate}[{\rm(i)}]
\item
\label{prop:monotonia}
 {\em{(Monotonicity)}} Let  $\varphi_{1},\,\varphi_2\colon \mathcal{G}_{\Dx}\to\R$ be such that $\varphi_1\leq \varphi_2$. Then we have 
$$\mathcal S^{\rm fd}_{k,i}( \varphi_1,a)\leq \mathcal S^{\rm fd}_{k,i}( \varphi_2,a) \quad \text{for all } (k,i)\in \mathcal{I}_{\Delta t}^*\times{\mathcal I}^{\circ}_{\Delta x},\, a\in A. $$
\item
\label{prop:constant_addition}
{\em{(Addition of a constant)}} Let $C\in \R$, $\varphi\colon \mathcal{G}_{\Dx}\to\R$, $(k,i)\in \mathcal{I}_{\Delta t}^*\times{\mathcal I}^{\circ}_{\Delta x}$, and $a\in A$. Then, if $\lambda^{\pm,\ell}_{k,i}(a)=1$ for all $\ell\in\I$, we have
$$
\mathcal S^{\rm fd}_{k,i}( \varphi+C,a)= S^{\rm fd}_{k,i}( \varphi,a)+C.
$$
Otherwise, it holds that 
\begin{align}
\mathcal S^{\rm fd}_{k,i}( \varphi+C,a)&\leq S^{\rm fd}_{k,i}( \varphi,a)+C\quad\text{if } C\geq 0,\nonumber\\
\mathcal S^{\rm fd}_{k,i}( \varphi+C,a)&\geq S^{\rm fd}_{k,i}( \varphi,a)+C\quad\text{if } C\leq 0.
\nonumber
\end{align}
\item
\label{prop:stability}
{\em{(Stability)}} If $\{V_{k}\colon\mathcal G_{\Dx}\to\R\,|\,k\in \mathcal I_{\Dt}\}$ solves~\eqref{nonlinear_fully}, then 
$$
\| V_{k}\|_{\infty} \leq  \|\Psi\|_{\infty} +T \|f\|_{\infty}\quad\text{for all }k\in\mathcal{I}_{\Delta t}.
$$ 
\end{enumerate}
\end{prop}

\begin{proof} 
\eqref{prop:monotonia}: Since $I[\varphi_{1}]\leq I[\varphi_{2}]$,  the assertion follows from~\eqref{operatore_con_controllo} using that 
$\gamma^{\pm,\ell}_{k,i}(a)$, and $\pi^\ell_{k,i}(a)$ are nonegative for all $ (k,i,\ell)\in \mathcal{I}_{\Delta t}^*\times {\mathcal I}^{\circ}_{\Delta x}\times\I$ and $a\in A$. 
\vspace{0.2cm}

\eqref{prop:constant_addition}: Follows directly from~\eqref{operatore_con_controllo}.

\eqref{prop:stability}: For every $(k,i)\in \mathcal{I}_{\Delta t}^*\times {\mathcal I}^{\circ}_{\Delta x}$ and $a\in A$, it follows from the definition of $\mathcal S^{\rm fd}_{k,i}(\cdot,a)$, the equalities  $\gamma_{k,i}^{+,\ell}(a)+\gamma_{k,i}^{-,\ell}(a)=1$, for all $\ell\in\mathcal{I}$, and $\sum_{\ell=1}^{p}\pi^{\ell}_{k,i}(a)=1$,  that  
$$
 \left|\mathcal S^{\rm fd}_{k,i}(V_{k+1},a)\right|
 \leq 
 \max(\| V_{k+1}\|_{\infty},\|\Psi\|_{\infty} )+ \tau_{k,i}(a) \|f(\cdot,\cdot,a)\|_{\infty}\leq \max(\| V_{k+1}\|_{\infty},\|\Psi\|_{\infty} )+ \Delta t \|f\|_{\infty}
$$
and hence, by~\eqref{nonlinear_fully},
$$
\|V_{k}\|_{\infty}\leq\max(\| V_{k+1}\|_{\infty},\|\Psi\|_{\infty} )+ \Delta t \|f\|_{\infty}.
$$
Finally, by iteration and using that $V_{N,i}=\Psi(T,x_{i})$ for all $i\in\mathcal{I}_{\Delta x}$, we get that $\|V_{k}\|_{\infty}\leq  \|\Psi\|_{\infty} + T \|f\|_{\infty}$ for all $k\in\mathcal{I}_{\Delta t}$. 
\end{proof}
In what follows  we set 
\begin{equation}
\label{eq:extension}
V_{\Delta t,\Delta x}(t,x):= I [V_{[t/\Delta t]}](x)\quad\text{for all }(t,x)\in\overline Q_T.
\end{equation}

Let $(\Delta t_{n}, \Delta x_n)\in (0,+\infty)^2$ be such that $\lim_{n\to\infty} (\Delta t_n,\Delta x_n)= (0,0)$ and $(\Delta x_n)^2/\Delta t_n \to 0$.  For every $n\in\NN$, let us set $V_n= V_{\Delta t_n,\Delta x_n}$ and define 
\be
\label{discrete_semi_limits}
\overline V(t,x) :=\underset{\substack{s\to t, y\to x\\ n\to\infty}}{\lim\sup} V_{n}(s,y)\quad \text{and} \quad 
\underline V(t,x) :=\underset{\substack{s\to t, y\to x\\ n\to\infty}}{\lim\inf} V_{n}(s,y)
\ee
for all $(t,x)\in \overline{Q}_{T}$. Notice that the stability of the scheme in Proposition~\ref{prop:properties}\eqref{prop:stability}  ensures that $\overline V(t,x)$ and $\underline V(t,x)$ are finite.
\subsection{Convergence in the case where the boundary conditions hold in the strong sense}
We begin by considering the convergence at the boundary points $(t,x)\in\partial^{*}Q_{T}$.
\begin{lem}
\label{prop:boundary_conv}
Assume that {\bf(H1)} hold. Then the following hold:
\begin{enumerate}[{\rm(i)}]
\item 
\label{prop:boundary_conv_i}
We have $\overline{V}(T,\cdot)=\underline{V}(T,\cdot)=\Psi(T,\cdot)$ on $\overline{\Omega}$. 
\item 
\label{prop:boundary_conv_ii}
Assume that in addition {\bf(H2)} holds. Then $\overline{V}=\underline{V}=\Psi$ on $\partial^* Q_T$.
\end{enumerate}  
\end{lem}
\begin{proof} Let us fix $n\in\NN$. In what follows, we write $\mathcal S^{\Psi,f}$ for  the operator $\mathcal S^{\rm fd}$, defined in \eqref{operatore_con_controllo}, to underline its  dependence on the boundary condition $\Psi$ and on the running cost $f$. Notice that, by definition, for all $\varphi\colon \mathcal G_{\Delta x_n} \to \R$, $(k,i)\in\I^*_{\Delta t_n}\times {\I}^{\circ}_{\Delta x_n}$, and $a\in A$, we have
\be
\label{monotonicity_scheme_boundary_running_terms}
\mathcal S^{\Psi_1,f_1}_{k,i}(\varphi,a)\leq \mathcal S^{\Psi_2,f_2}_{k,i}(\varphi,a) \quad \text{if $\Psi_1\leq \Psi_2$ and $f_1\leq f_2$.}
\ee
{\rm(i)} By the Tietze extension theorem, there exists a continuous extension of $\Psi$ to $[0,T]\times \R^d$, which, for notational convenience, we still denote by $\Psi$. Let us fix $\overline{\varepsilon}>0$ and denote by $\omega$ a strictly increasing modulus of continuity of $\Psi$ on the set $\{(t,y)\in [0,T]\times \R^d \, | \, \text{dist}(y,\Omega) \leq \overline{\varepsilon}\}$.  Let $\phi\in C^{\infty}(\R^d)$ be  non-negative, supported in the unit ball in $\R^d$, and such that  $\int_{\R^d}\phi(x){\rm d} x=1$. Given $\varepsilon \in (0,\overline{\varepsilon})$, let us set $\phi_{\varepsilon}(\cdot)=\frac{1}{\varepsilon^d}\phi(\cdot/{\varepsilon})$ and define $\Psi^\varepsilon_T\colon\overline{\Omega} \to \R$ by $\Psi^\varepsilon_T(x) = (\phi_{\varepsilon}*\Psi(T,\cdot))(x) + 2\omega(\varepsilon)$. By definition,  one has
\begin{equation}
\label{differenza_Psi_eps}
0<\omega(\vare) \leq \Psi^\varepsilon_T(x) - \Psi(T,x)   \leq 3 \omega(\vare) \quad \text{for all } x\in \overline{\Omega}.
\end{equation}
Given $K_1>0$, define $v^\vare\colon \overline{Q}_{T}\to \R$ by
\begin{equation}
\label{definizione_v_bar}
 v^{\vare}(t,x) : = K_1 (T-t) + \Psi^\varepsilon_T(x) \quad \text{for all } (t,x)\in\overline{Q}_{T},
\end{equation}
which, for $K_1=K_1(\varepsilon)$ large enough, is a supersolution of \eqref{eq:HJB} in the strong sense. Indeed, if $(t,x)\in Q_{T}$ then \eqref{definizione_v_bar}, the continuity of the coefficients, and the compactness of $A$  imply that 
\be
\label{supersolution_v_vare}
-\partial_t v^{\vare}(t,x)+\sup_{a\in A} \mathcal{L}^{a}(t, x,D_x v^{\vare}(t,x), D^2_{x} v^{\vare}(t,x)) \geq 0
\ee
 for $K_1$ large enough. Now, suppose that $(t,x)\in \partial^* Q_{T}$. Conditions \eqref{differenza_Psi_eps} and \eqref{definizione_v_bar} imply that, if $t=T$, then   $v^{\vare}(T,x)>\Psi(T,x)$ for all $x\in \overline{\Omega}$, while if $t<T$,  by taking $K_1>\omega(T)/\vare$, for all $x\in \partial \Omega$ one has  
$$
v^{\vare}(t,x) \geq  \frac{\omega(T)}{\vare}(T-t) + \Psi(T,x)  + \omega(\vare) \geq \frac{\omega(T)}{\vare}(T-t) + \Psi(t,x)  - \omega(T-t) +  \omega(\vare). 
$$
In turn, since $\omega$ is increasing, we deduce that 
\begin{equation}
\label{veps_psi_bordo}
v^{\vare} (t,x) > \Psi(t,x)  \quad \text{for all } (t,x) \in \partial^* Q_{T}.
\end{equation}
Next, let us show that for any $(k,i)\in {\mathcal{I}}_{\Delta t_n}\times {\mathcal{I}}_{\Delta x_n}$ one has
\begin{equation}
\label{eq:uppervbar}
V_{k,i} \leq v^{\vare}(t_k,x_i)+e_{1,n},
\end{equation}
where $e_{1,n}\to 0$  as $n\to\infty$. For $(k,i)\in (\mathcal I^*_{\Dt_n}\times \partial{\mathcal{I}}_{\Dx_n})\cup (\{N\}\times {\mathcal{I}}_{\Delta x_n})$, inequality~\eqref{veps_psi_bordo} implies $V_{k,i}=\Psi(t_k,x_i) \leq v^{\vare}(t_k,x_i)$. Now, for $(k,i)\in\mathcal{I}^{*}_{\Delta t_n} \times {\mathcal{I}}^{\circ}_{\Delta x_n}$ inequalities~\eqref{veps_psi_bordo} and~\eqref{monotonicity_scheme_boundary_running_terms}  yield
\begin{align}
\label{eq:Sbar}
 \underset{a\in A}\sup \left(\frac{v^\vare(t_k,x_i) -{\mathcal S}_{k,i}^{\Psi,f}(v^\vare_{k+1},a)}{\tau_{k,i}(a)}\right)\geq  \underset{a\in A}\sup \left(\frac{v^\vare(t_k,x_i) -{\mathcal S}_{k,i}^{v^\vare,f}(v^\vare_{k+1},a)}{\tau_{k,i}(a)}\right).
\end{align}
Then, by~\eqref{supersolution_v_vare} and  Proposition~\ref{consistenza_caso_nonlineare} applied to ${\mathcal S}^{v^\vare,f}$, we get
\begin{equation}
\label{eq:sup_schemes}
\begin{array}{l} 
  \underset{a\in A}\sup \left(\frac{ v^\vare(t_k,x_i) -\mathcal S_{k,i}^{\Psi,f}(v^\vare_{k+1},a)}{\tau_{k,i}(a)}\right)\\ 
 \geq  \underset{a\in A}\sup \left(\frac{v^\vare(t_k,x_i) -{\mathcal S}_{k,i}^{v^\vare,f}(v^\vare_{k+1},a)}{\tau_{k,i}(a)}\right) - \left(-\partial_{t} v^\vare(t_k,x_i) + \underset{a\in A}\sup\,\mathcal{L}^a(t_k,x_i, D_x v^\vare(t_k, x_i) , D_{x}^2 v^\vare(t_k,x_i))\right) \\[3pt]
  \geq  -\delta_n,
\end{array}
\end{equation}
where $0\leq \delta_n \to 0$ as $n\to \infty$. Now, let $\hat{a}\in A$ be such that 
$$
\underset{a\in A}\sup \left(\frac{ v^\vare(t_k,x_i) -\mathcal S_{k,i}^{\Psi,f}(v^\vare_{k+1},a)}{\tau_{k,i}(a)}\right) \leq   \frac{ v^\vare(t_k,x_i) -\mathcal S_{k,i}^{\Psi,f}(v^\vare_{k+1},\hat{a})}{\tau_{k,i}(\hat{a})} +\delta_n.
$$
It follows from \eqref{eq:sup_schemes} that 
\begin{equation}
\label{eq:ineqbar}
 v^\vare (t_k,x_i) \geq  \mathcal S_{k,i}^{\Psi,f}(v^\vare_{k+1},\hat{a}) - 2\tau_{k,i}(\hat{a})\delta_n\geq
 \underset{a\in A}\inf\mathcal S_{k,i}^{\Psi,f}(v^\vare_{k+1},a) - 2\delta_n \Delta t_n.
\end{equation}
In particular, by taking $k=N-1$, using that $v^\vare(T,x_j)\geq \Psi(T,x_j)= V_{N,j}$ for all $j\in\mathcal{I}_{\Delta x_n}$, Proposition~\ref{prop:properties}\eqref{prop:monotonia} yields 
$$
\begin{array}{rcl}
 v^\vare (t_{N-1},x_i) &\geq&   \underset{a\in A}\inf\mathcal S_{N-1,i}^{\Psi,f}(v^\vare_{N},a) - 2\delta_n\Delta t_n \\[10pt]
\; & \geq &  \underset{a\in A}\inf\mathcal S_{N-1,i}^{\Psi,f}(V_{N},a) - 2\delta_n\Delta t_n \\[10pt]
 \; &=& V_{N-1,i} - 2\delta_n\Delta t_n.
\end{array}
$$
By taking $k=N-2$, using~\eqref{eq:ineqbar} and Proposition~\ref{prop:properties}\eqref{prop:constant_addition} we have 
$$
\begin{array}{rcl}
 v^\vare (t_{N-2},x_i) &\geq&   \underset{a\in A}\inf\mathcal S_{N-2,i}^{\Psi,f}(v^\vare_{N-1},a) - 2\delta_n \Delta t_n \\[10pt]
\; & \geq &  \underset{a\in A}\inf\mathcal S_{N-2,i}^{\Psi,f}( V_{N-1} - 2\delta_n \Delta t_n,a) - 2\delta_n \Delta t_n \\[10pt]
\; & \geq & \underset{a\in A}\inf\mathcal S_{N-2,i}^{\Psi,f}( V_{N-1},a) - 4\delta_n \Delta t_n\\[10pt]
 \; &= & V_{N-2,i} - 4\delta_n \Delta t_n.
\end{array}
$$
Proceeding in this manner for $k=N-3,...,0$ we get that \eqref{eq:uppervbar} holds with  $e_{1,n} =2T\delta_{n}$. 

On the other hand, given $K_{2}>0$ and $\vare\in (0,\overline{\vare})$, defining  $\tilde{\Psi}^\varepsilon_T:\overline{\Omega} \to \R$ by $\tilde{\Psi}^\varepsilon_T(x) = (\phi_{\varepsilon}*\Psi(T,\cdot))(x)- 2\omega(\varepsilon)$ and   $v_\vare\colon\overline{Q}_{T}\to \R$ by
\begin{equation}
\label{definizione_v_bar_bis}
 v_{\vare}(t,x) : = -K_2(T-t) + \tilde{\Psi}^\varepsilon_T(x) \quad \text{for all } (t,x)\in\overline{Q}_{T},
\end{equation}
 if $K_{2}=K_{2}(\vare)$ is large enough, we can argue as before to get the existence of $e_{2,n}$ such that  $e_{2,n}\to 0$  as $n\to\infty$, and
\begin{equation}
\label{inequality_veps_due}
  v_{\vare}(t_k,x_i)-e_{2,n}\leq V_{k,i}  \quad \text{for all } (k,i) \in \I_{\Delta t_n}\times \mathcal{I}_{\Delta x_n}.
\end{equation}
Thus, from \eqref{eq:uppervbar}, \eqref{inequality_veps_due}, and \eqref{eq:extension}, for every $(s,y)\in \overline{Q}_{T}$, we have that 
$$
-\|v_{\vare}(s,\cdot)-I[v_{\vare}(s,\cdot)]\|_{\infty}+v_\vare(s,y)-e_{2,n}\leq V_{n}(s,y) \leq  \|v^{\vare}(s,\cdot)-I[v^{\vare}(t,\cdot)]\|_{\infty}+ v^{\vare}(s,y)+e_{1,n}.
$$
Considering a sequence $(s_n,y_n)$ converging  to $(T,x)$, recalling~\eqref{interp_estim}, and taking the limit $n\to\infty$ we obtain 
\begin{equation}
\label{three_inequalities}
\tilde{\Psi}^\varepsilon_T(x) \leq  \underline{V}(T,x) \leq   \overline{V}(T,x)\leq \Psi^\varepsilon_T(x)
\end{equation}
and hence {\rm(i)} follows by letting $\vare\to 0$ in \eqref{three_inequalities}.

{\rm(ii)} We proceed in two steps.\\
{\it \underline{Step 1}{\rm:} Assume $f\geq 0$, $\Psi\geq 0$, and $\Psi=0$ in $[0,T]\times\partial \Omega$}. Let $\delta>0$ and let $\zeta \in C^2(\overline{\Omega}_{\delta})$ be as in {\bf (H2)}(i). Given $K,\,\theta>0$, let us set
\begin{equation}
\label{def:w}
w(x) := K \zeta(x) +\theta\quad\text{for all }x\in \overline{\Omega}_{\delta}.
\end{equation}
Let us show that one can always choose $K>0$ large enough to ensure that
\begin{itemize}
\item[(a)] $\sup_{a\in A}\mathcal{L}^a(t,x, D_x w(x) , D_{x}^2 w(x)) \geq 0 $ for all $(t,x)\in[0,T)\times  \Omega_\delta$,
\item[(b)] $w(x)>\Psi(t,x)$ for all $(t,x)\in\big(\{T\}\times\overline{\Omega}_{\delta}\big)\cup \big([0,T)\times\partial\Omega\big)$,
\item[(c)] $w(x)>\sup\{\|V_{k}\|_{\infty}\,|\,k\in\I_{\Delta t_n}\}$ for all $x\in\overline \Omega_\delta\setminus \Omega_{\delta/2}$.
\end{itemize}

Indeed, it follows from~\eqref{super_solution_barrier} that, if $K> \|f\|_\infty/\eta$, then {\rm(a)} holds. On the other hand, if $(t,x)\in [0,T)\times \partial \Omega$, then $w(x)=\theta>0=\Psi(t,x)$. If $(t,x)\in\{T\}\times\overline{\Omega}_{\delta}$, since $\Psi(T,x)=0$ and $\zeta(x)=0$ if $x\in\partial\Omega$, we have $\lim_{y\to x}\frac{\Psi(T,y)-\theta}{\zeta(y)}=-\infty$ and, hence, $\mathfrak{s}:=\sup_{y\in \overline \Omega_\delta\setminus\partial \Omega} \frac{\Psi(T,y)-\theta}{\zeta(y)}<+\infty$. Thus, {\rm(b)} holds with $K>\mathfrak{s}$. Finally, setting $M:= \|\Psi\|_{\infty} +T \|f\|_{\infty}$ and taking $K> \frac{(M-\theta)^+}{\inf_{x\in \overline \Omega_\delta\setminus \Omega_{\delta/2}} \zeta(x)}$, we obtain that, for every $x\in \overline \Omega_\delta\setminus \Omega_{\delta/2}$, 
$w(x)>M$ and assertion {\rm(c)} follows from Proposition~\ref{prop:properties}(iii). 

Set $\mathcal{I}_{\Delta x}^{\circ,\delta}=\{i\in \I_{\Delta x}^{\circ}\,|\, x_{i}\in \Omega_{\delta/2}\}$ and let $n\in\NN$ be large enough in order to ensure that $y_{k,i}^{\pm,\ell}(a)\in\overline{\Omega}_{\delta}$ for any $k \in \I_{\Delta t_{n}}$ and $i\in \I_{\Delta x_{n}}^{\circ,\delta}$, $a\in A$, and $\ell\in\I$. Arguing as in the proof of Proposition~\ref{consistenza_caso_nonlineare}, it holds that 
\begin{equation}
\label{def:eta_n}
\eta_{n}:=\sup_{k\in\mathcal{I}^{*}_{\Delta t},\,i\in \I_{\Delta x_{n}}^{\circ,\delta}}\Bigg| \sup_{a\in A}\Bigg(\frac{w(x_{i})-\mathcal{S}_{k,i}^{w,f}(w,a)}{\tau_{k,i}(a)}\Bigg) -\sup_{a\in A}\mathcal{L}^{a}(t_{k},x_{i},D_{x}w(x_{i}),D_{x}^2 w(x_{i}))\Bigg|\underset{n\to\infty}{\longrightarrow} 0,
\end{equation}
Let us show that 
\begin{equation}
\label{eq:backward_consistency_w}
V_{k,i} \leq w(x_i) + 2\eta_{n}T \quad\text{for all }k\in\I_{\Delta t_{n}},\, i\in\mathcal{I}_{\Delta x}^{\circ,\delta}\cup \partial{\mathcal{I}}_{\Dx_n}.
\end{equation}
Indeed, if $(k,i)\in \big(\I_{\Delta t_{n}}^{*}\times \partial{\mathcal{I}}_{\Dx_n}\big)\cup\big(\{N\}\times (\mathcal{I}_{\Delta x_n}^{\circ,\delta}\cup \partial{\mathcal{I}}_{\Dx_n})\big)$,~\eqref{eq:backward_consistency_w} follows from property (b) above and~\eqref{nonlinear_fully}. On the other hand, for every $(k,i)\in\I_{\Delta t_{n}}^{*}\times \mathcal{I}_{\Delta x_n}^{\circ,\delta}$, the positivity of the weights $\gamma_{k,j}^{\pm,\ell}(a)$ and property (b) yield ${\mathcal S}_{k,i}^{w,f}(w,a)\geq {\mathcal S}_{k,i}^{\Psi,f}(w,a)$ for all $a\in A$ and, hence, 
\begin{align}\label{eq:Sbarspace}
 \underset{a\in A}\sup \left(\frac{w(x_i) - {\mathcal S}_{k,i}^{w,f}(w,a)}{\tau_{k,i}(a)}\right) \leq  \underset{a\in A}\sup \left(\frac{w(x_i) -{\mathcal S}_{k,i}^{\Psi,f}(w,a)}{\tau_{k,i}(a)}\right).
\end{align}
In turn, property (a) and~\eqref{def:eta_n} imply that 
\begin{equation}
\underset{a\in A}\sup \left(\frac{w(x_i) -\mathcal S_{k,i}^{\Psi,f}(w,a)}{\tau_{k,i}(a)}\right) \geq  \underset{a\in A}\sup \left(\frac{w(x_i) -{\mathcal S}_{k,i}^{w,f}(w,a)}{\tau_{k,i}(a)}\right) - \underset{a\in A}\sup\,\mathcal{L}^a(t_k,x_i, D_x w , D_{x}^2 w)    \geq  - \eta_{n}
\label{eq:consistency_w}
\end{equation}
and, since there exists $\widehat{a}\in A$ such that 
$$
\frac{w(x_i) -\mathcal S_{k,i}^{\Psi,f}(w,\widehat{a})}{\tau_{k,i}(\widehat{a})}+\eta_{n}\geq \underset{a\in A}\sup \left(\frac{w(x_i) -\mathcal S_{k,i}^{\Psi,f}(w,a)}{\tau_{k,i}(a)}\right),
$$
we obtain that 
\begin{equation}
 w(x_i)+2\eta_{n}\Delta t_{n} \geq  \underset{a\in A}\inf\mathcal S_{k,i}^{\Psi,f}(w,a).
\label{eq:ineqbarspace}
\end{equation}
Thus, for every $i\in \mathcal{I}_{\Delta x}^{\circ,\delta}$, it follows form property (b) that  
$$
V_{N-1,i}=\underset{a\in A}\inf\mathcal S_{N-1,i}^{\Psi,f}(V_{N},a)=\underset{a\in A}\inf\mathcal S_{N-1,i}^{\Psi,f}(\Psi(T,\cdot),a)\leq\underset{a\in A}\inf\mathcal S_{N-1,i}^{\Psi,f}(w,a)\leq  w(x_i)+2\eta_{n}\Delta t_{n}
$$
and hence, using property (c), Proposition~\ref{prop:properties}\eqref{prop:monotonia},\eqref{prop:constant_addition}, and~\eqref{eq:ineqbarspace} we get 
$$
V_{N-2,i}=\underset{a\in A}\inf\mathcal S_{N-2,i}^{\Psi,f}(V_{N-1},a)\leq \underset{a\in A}\inf\mathcal S_{N-2,i}^{\Psi,f}(w,a) +2\eta_{n}\Delta t_{n}\leq w(x_{i})+4\eta_{n}\Delta t_{n}.
$$
Arguing in this manner for $k=N-3,\hdots,0$, we obtain~\eqref{eq:backward_consistency_w}.

Finally, for every $(s,y)\in [0,T]\times \overline\Omega_{\delta/4}$ and $n$ large enough, one has 
$$
0\leq V_{n}(s,y)=  I [V_{[s/{\Delta t_n}],\cdot}](y) \leq  I [w](y) + 2\eta_{n}T \leq w(y) +\|w-I[w]\|_{\infty} + 2\eta_{n}T$$
and, hence, we deduce from~\eqref{def:w} and~\eqref{interp_estim} that, for every $(t,x)\in[0,T)\times\partial\Omega$, 
$$
0 \leq  \underline{V}(t,x) \leq  \overline{V}(t,x)\leq \theta
$$
and, since $\theta>0$ is arbitrary, we obtain that
$$
\underset{\substack{s\to t, y\to x\\ n\to \infty}}{\lim} V_{n}(s,y) = 0.
$$

{\it \underline{Step 2}: The general case.} Let $g$ be as in {\bf(H2)}{\rm(ii)} and let $v$ be the unique viscosity solution, in the strong sense, to~\eqref{eq:HJB}. Let us set
\begin{align}
\tilde f&:= f+\partial_t g+b^{\top}D_x g+\frac12 \text{Tr}[\sigma\sigma^\top  D^2_x g]\quad\text{in } \overline{Q}_{T}\times A,
\label{def:tilde_f}\\
\tilde \Psi&:=\Psi-g\quad\text{in }\partial^{*}Q_T.
\label{def:tilde_Psi}
\end{align}
By definition, $\tilde f\geq 0$ on $\overline Q_T\times A$, $\tilde \Psi\geq 0$ on $\partial^* Q_T$ and $\tilde \Psi=0$ on $[0,T]\times\partial \Omega$.
Let us denote by $\tilde v$ and $\tilde V_{n}$, respectively, the solution to~\eqref{eq:HJB} and the solution to the scheme~\eqref{nonlinear_fully} obtained after replacing $f$ by $\tilde f$ and $\Psi$ by $\tilde \Psi$. It follows from Proposition~\ref{prop:weak_comparison} that $v=\tilde v+g$ and, by Step 1, for every $(t,x)\in [0,T)\times \partial\Omega$ we have 
\begin{equation}\label{eq:limtilde}
\underset{\substack{(s,y)\to (t,x)\\ n\to \infty}}{\lim} \tilde V_{n}(s,y) = 0.
\end{equation}
Now, let us set
\begin{multline}
\label{def:eta_n_tilde}
\tilde{\eta}_n:=\sup_{k\in\mathcal{I}^{*}_{\Delta t_n},\,i\in {\mathcal{I}}^{\circ}_{\Delta x_n},\,a\in A}\Bigg| \frac{g(t_k,x_i) -  \mathcal S^{g,0}_{k,i}(g^{k+1}, a)}{\tau_{k,i}(a)} \\
 -\left(  -\partial_t g(t_k,x_i) - \langle b(t_k,x_i,a),D_x g(t_k,x_i)\rangle -\frac12 \text{Tr}[\sigma\sigma^\top(t_k,x_i,a)  D^2_x g(t_k,x_i)] \right) \Bigg|
\end{multline}
and let us show that 
\begin{equation}
\label{eq:Vg}
\sup_{k\in\mathcal{I}_{\Delta t_n},\,i\in \mathcal{I}_{\Delta x_n}}\left|(V_n)_{k,i}-\big((\tilde{V}_n)_{k,i}+g(t_k,x_i)\big) \right| \leq (N-k)\Delta t_n\tilde{\eta}_n.
\end{equation}
Indeed, by~\eqref{def:tilde_Psi}, we have that
\begin{equation}
\label{eq:equality_bordo}
(V_{n})_{k,i}=(\tilde  V_{n})_{k,i}+g(t_k,x_i)\quad\text{for all }(k,i) \in\big(\I_{\Delta t_{n}}^{*}\times \partial{\mathcal{I}}_{\Dx_n}\big)\cup\big(\{N\}\times  \mathcal{I}_{\Dx_n}\big).
\end{equation}
On the other hand, for every $(k,i)\in\mathcal{I}^{*}_{\Delta t_n}\times{\mathcal{I}}^{\circ}_{\Delta x_n}$,~\eqref{operatore_con_controllo},  \eqref{def:tilde_f}, and~\eqref{def:eta_n_tilde} yield
\begin{align}
(\tilde{V}_n)_{k,i}&=\inf_{a\in A}\mathcal S_{k,i}^{\tilde \Psi,\tilde f}\big((\tilde{V}_n)_{k+1},a\big)\nonumber\\
& \leq  \inf_{a\in A} \left(\mathcal S_{k,i}^{\tilde \Psi,f}\big((\tilde{V}_n)_{k+1},a\big) - g(t_k,x_i) + \mathcal S^{g,0}_{k,i}(g_{k+1}, a) + \tau_{k,i}(a)\tilde{\eta}_{n}\right)\nonumber\\
& =  \inf_{a\in A}\left(\mathcal S_{k,i}^{\Psi,f}((\tilde{V}_n+g)_{k+1},a)+ \tau_{k,i}(a)\tilde{\eta}_{n}\right)- g(t_k,x_i) \nonumber\\
&\leq  \inf_{a\in A}\mathcal S_{k,i}^{\Psi,f}((\tilde{V}_n+g)_{k+1},a) - g(t_k,x_i)+\Delta t_{n}\tilde{\eta}_{n}. 
\label{eq:leq_inequality_operator}
\end{align}
Now, assume that, for $k\in\I_{\Delta t_n}^{*}$, we have $(\tilde{V}_n)_{k+1}\leq (V_n)_{k+1}-g_{k+1}+(N-k-1)\Delta t_{n}\tilde{\eta}_{n}$. Then it follows from~\eqref{eq:leq_inequality_operator} and Proposition~\ref{prop:properties}\eqref{prop:monotonia},\eqref{prop:constant_addition} that
\begin{align}
(\tilde{V}_n)_{k,i} & \leq  \inf_{a\in A}\mathcal S_{k,i}^{\Psi,f}\Big((V_n)_{k+1} - g_{k+1} +(N-k-1)\Delta t_n \tilde{\eta}_{n} + g_{k+1},a\Big) - g(t_k,x_i)+ \Delta t_n\tilde{\eta}_{n}\nonumber\\
&\leq\inf_{a\in A}\mathcal S_{k,i}^{\Psi,f}((V_n)_{k+1},a) - g(t_k,x_i)+(N-k-1)\Delta t_n \tilde{\eta}_{n}+\Delta t_n\tilde{\eta}_{n}\nonumber\\
& = (V_n)_{k,i}-g(t_k,x_i)+(N- k)\Delta t_n \tilde{\eta}_{n}.
\label{eq:inequality_one_side}
\end{align}
Analogously, arguing as in~\eqref{eq:leq_inequality_operator}, for every $(k,i)\in\mathcal{I}^{*}_{\Delta t_n}\times{\mathcal{I}}^{\circ}_{\Delta x_n}$ we have 
\begin{align*}
(\tilde{V}_n)_{k,i} & \geq  \inf_{a\in A}  \mathcal S_{k,i}^{\Psi,f}((\tilde{V}_n+g)_{k+1},a) - g(t_k,x_i)- \Delta t_{n}\tilde{\eta}_{n}
\end{align*} 
and hence, if $(\tilde{V}_n)_{k+1}\geq (V_n)_{k+1}-g_{k+1}-(N-k-1)\Delta t_{n}\tilde{\eta}_{n}$ for some $k\in\I_{\Delta t_n}^{*}$, arguing as in~\eqref{eq:inequality_one_side} we have  
\begin{equation}
\label{eq:inequality_other_side}
(\tilde{V}_n)_{k,i}\geq (V_n)_{k,i}-g(t_k,x_i)-(N- k)\Delta t_n \tilde{\eta}_{n}.
\end{equation}
Altogether,~\eqref{eq:Vg} follows from~\eqref{eq:equality_bordo},~\eqref{eq:inequality_one_side}, and~\eqref{eq:inequality_other_side}. In turn, setting $\mathcal{E}_n=\sup_{(t,x)\in \overline{Q}_{T}}\big|g(t,x)-I[g_{[t/\Delta t_n]}](x)\big|$, it follows from~\eqref{eq:extension} that, for every $(s,y)\in\overline{Q}_{T}$, one has
\begin{equation}
\label{eq:estimate_differrence_value_functions}
\left| \tilde{V}_{n}(s,y)- (V_{n}(s,y) - g(s,y))\right| \leq (N-k)\Delta t_n\tilde{\eta}_{n}+ \mathcal E_n\leq T\tilde{\eta}_{n}+ \mathcal E_n.
\end{equation}
Since Lemma~\ref{lem:consistency_aux} and ~\eqref{interp_estim} imply that $\tilde{\eta}_{n}\to 0$ and $\mathcal E_n\to 0$, respectively, we deduce from~\eqref{eq:limtilde} that, for every $(t,x)\in [0,T)\times \partial\Omega$, we have 
\begin{equation}\label{eq:limtilde_final}
\underset{\substack{(s,y)\to (t,x)\\ n\to \infty}}{\lim}  V_{n}(s,y) = g(t,x)=\Psi(t,x),
\end{equation}
from which the result follows.
\end{proof}

\begin{theo}[Convergence to the value function in the strong framework]
Assume that {\bf(H1)} and {\bf(H2)} hold. Moreover, suppose that $\partial \Omega $ is of class $C^2$ and that, as $n\to\infty$, $(\Delta x_n)^2/\Delta t_n\to 0$. Then  
\begin{equation} 
\label{eq:uniform_convergence_value_function}
V_{n}\underset{n\to\infty}{\longrightarrow} V \quad \text{uniformly in $\overline{Q}_T$},
\end{equation}
where $V$ is the value function defined in~\eqref{value_function_exit_from_open}.
\label{theo:main_result_1}
\end{theo}
\begin{proof}
We follow the classical approach introduced in \cite{BS91}.  Notice that, by~\cite[Lemma 1.5 in Chapter V]{MR1484411}, $\overline{V}$ and $\underline{V}$, defined in~\eqref{discrete_semi_limits}, are upper and lower semicontinuous, respectively. Thus, thanks to  Proposition \ref{prop:properties} and standard techniques (see e.g.~\cite{CamFal95,DebrJako13}), one has that $\overline{V}$ and $\underline{V}$ are sub- and supersolutions to~\eqref{parabolic_HJB} in $Q_T$, respectively. In turn, it follows from Lemma~\ref{prop:boundary_conv} that $\overline{V}$ and $\underline{V}$ are, respectively, sub- and supersolutions to~\eqref{eq:HJB} in $\overline{Q}_T$ in the strong sense of Definition~\ref{strong_boundary_conditions}. Therefore, by Proposition~\ref{prop:weak_comparison}, one has $\overline{V}\leq\underline{V}$ in $\overline{Q}_{T}$ and, since the reverse inequality always holds, we deduce that $\overline{V}=\underline{V}$ in $\overline{Q}_{T}$. It follows from Proposition~\ref{prop:value_function_unique_viscosity_sol} that  $\overline{V}=\underline{V}=V$ in $\overline{Q}_{T}$ and, by~\cite[Lemma 1.9 in Chapter V]{MR1484411}, we obtain that~\eqref{eq:uniform_convergence_value_function} holds.
\end{proof}
\vspace{-0.2cm}
\begin{rem} Error estimates for general monotone numerical schemes are provided in \cite{PRR_error} by means of the Krylov's ``shaking coefficient'' regularization technique. Here, if compared with {\bf(H1)}-{\bf(H2)}, stronger assumptions on the dynamics are needed in order to ensure the H\"older continuity of the value function necessary to perform the aforementioned regularization step.
\end{rem}

\subsection{Convergence in the case where the boundary conditions hold in the weak sense}\label{subsec:weak}

As discussed in Section~1, in general the value function $V$ is not guaranteed to be continuous and hence it satisfies equation \eqref{eq:HJB} in the weak sense of Definition~\ref{weak_boundary_conditions}. The next result shows that, in this case, convergence still holds in $Q_{T}$ provided that the strong comparison principle in Definition~\ref{strong_comparison_principle} holds.

\begin{theo}[Convergence to the value function in the weak framework] Suppose that {\bf(H1)} and the strong comparison principle hold. Then, if $(\Delta x_{n})^{2}/\Delta t_n\to 0$ as $n\to\infty$, we have
\begin{equation} 
\label{eq:uniform_convergence_value_function_weak_framework}
V_{n}\underset{n\to\infty}{\longrightarrow} V \quad \text{locally uniformly in $[0,T]\times\Omega$},
\end{equation}
where $V$ is the value function defined in~\eqref{value_function_exit_from_open}.
\label{theo:main_result_2}
\end{theo}
\begin{proof}
Following the strategy in~\cite{BS91}, one needs to show that $\overline V$ (resp. $\underline V$), defined in~\eqref{discrete_semi_limits}, is a viscosity subsolution (resp. supersolution) to \eqref{eq:HJB}, in the weak sense of Definition~\ref{weak_boundary_conditions}. Then one can use the strong comparison principle in Definition~\ref{strong_comparison_principle} to deduce~\eqref{eq:uniform_convergence_value_function_weak_framework}. Since the proof of the supersolution property is similar to the one for the subsolution, we only provide the proof of the latter. 

Fix $(t,x)\in\overline Q_T$. If $t=T$, it follows from Lemma~\ref{prop:boundary_conv}\eqref{prop:boundary_conv_i} that $\overline V(t,x) =\Psi(t,x)$. Now, suppose that $t<T$  and let $\phi\in C^\infty(\overline Q_T)$ be such that $(t,x)$ is a local maximum point of $(\overline V-\phi)$. Without any loss of generality, we may assume that $(t,x)$ is a strict local maximum point and that $\overline V(t,x)=\phi(t,x)$. Then, by standard arguments in the theory of viscosity solutions (see e.g.~\cite[Chapter V]{MR1484411}), there exists a sequence $(s_n,y_n)\in \overline{Q}_{T}$ converging to $(t,x)$ such that, up to some subsequence, $(s_n,y_n)$ are global maximum points for 
 $(V_{n}-\phi)$ and $V_{n} (s_n,y_n)\to \overline V(t,x)$. Then, for every $(s,y)\in\overline{Q}_{T}$, we have  
 \be 
 \label{eq:dis}
 V_{n} (s,y)\leq \phi(s,y)+\xi_n,\quad\text{where}\quad\xi_n:= V_{n} (s_n,y_n)-\phi(s_n,y_n)\underset{n\to\infty}{\longrightarrow} 0.
 \ee
If $x\in \Omega$, the viscosity subsolution property of $\overline V$  follows by standard arguments using the monotonicity and consistency of the scheme (see e.g.~\cite{CamFal95,DebrJako13}).   
Thus, let us assume that $x\in\partial\Omega$.  If  $\overline V(t,x)\leq  \Psi(t,x)$, then condition {\rm(i.2)} in Definition~\ref{weak_boundary_conditions} trivially holds. Therefore, suppose that $\overline V(t,x)>\Psi(t,x)$.  Let  $k\colon \mathbb N\to \{0,\ldots, N-1\}$ be such that $s_{n}\in[t_{k(n)}, t_{k(n)+1})$ and let $\mathcal{J}_{n}=\{i\in \I_{\Delta x_n}\,|\,\psi_{i}(y_{n})\neq 0\}$. Notice that the cardinal $|\mathcal{J}_{n}|$ of $\mathcal{J}_{n}$ is at most $d+1$ and hence there exists $M\in\{1,\hdots,d+1\}$ such that, up to some subsequence, $|\mathcal{J}_{n}|=M$ for all $n\in\NN$. In particular, $\mathcal{J}_{n}$ can be written as $\mathcal{J}_{n}=\{i_{n}^{1},\hdots,i_{n}^{M}\}$ with $i_{n}^{q}\in\mathcal{I}_{\Delta x_n}$ for all $q=1,\hdots,M$. Let $(\overline{\psi}_{1},\hdots,\overline{\psi}_{M})$ be a limit point of $ (\psi_{i_{n}^{1}}(y_{n}),\hdots,\psi_{i_{n}^{M}}(y_n))$. By extracting a subsequence, we can assume that $(\psi_{i_{n}^{1}}(y_{n}),\hdots,\psi_{i_{n}^{M}}(y_n))\to (\overline{\psi}_{1},\hdots,\overline{\psi}_{M})$ and, hence, $\overline{\psi}_{q}\geq 0$, for all $q=1,\hdots,M$, and $\sum_{q=1}^{M}\overline{\psi}_{q}=1$. Set $\mathcal{I}_{+}=\{q\in\{1,\hdots,M\}\,|\, \overline{\psi}_{q}>0\}$ and let $\overline{a}\in\text{argmax}_{a\in A}\mathcal{L}^{a}(t, x,D_x\phi(t,x) , D_x^2\phi(t,x))$.

Suppose that there exists a subsequence, still labeled by $n$, such that for all $q\in\I_{+}$ and $n\in\NN$ large enough, there exists $(o,\ell)\in\{+,-\}\times\I$ such that $\lambda_{k(n),i_{q}^{n}}^{o,\ell}(\overline{a})<1$.  Then there exist  subsets $I_{1}^{{\rm out}},\hdots,I_{M}^{{\rm out}}$ of $\{+,-\}\times \I$, with $I_{q}^{{\rm out}}\neq \emptyset$ for all $q\in\I_{+}$, such that, up to some subsequence, for every $q\in\{1,\hdots,M\}$ and $(o,\ell)\in\{+,-\}\times\I$, $\lambda_{k(n),i_{n}^{q}}^{o,\ell}(\overline{a})<1$ if and only if $(o,\ell)\in I_{q}^{{\rm out}}$. Setting $\mathcal{C}^{{\rm out}}:=\{(q,o,\ell)\in \{1,\hdots,M\}\times\{+,-\}\times\I\,|\,(o,\ell)\in I_{q}^{{\rm out}}\}$ and $\mathcal{C}^{{\rm in}}:=(\{1,\hdots,M\}\times\{+,-\}\times \I)\setminus \mathcal{C}^{{\rm out}}$, it follows from~\eqref{eq:extension} and ~\eqref{nonlinear_fully} that
\begin{align}
V_{n}(s_{n},y_{n})&= \sum_{q=1}^{M}\psi_{i_{n}^{q}}(y_{n})V_{k(n),i_{n}^{q}}\nonumber\\
&\leq \sum_{(q,o,\ell)\in\mathcal{C}^{{\rm out}}}\psi_{i_{n}^{q}}(y_{n})\pi_{k(n),i_{n}^{q}}^{\ell}(\overline{a})\gamma_{k(n),i_{n}^{q}}^{o,\ell}(\overline{a})\Psi\big(t_{k(n)}+\lambda_{k(n),i_{n}^{q}}^{o,\ell}(\overline{a})\Delta t,y_{k(n),i_{n}^{q}}^{o,\ell}(\overline{a})\big) 
\label{eq:v_n_convex_prelimit}\\
+\sum_{(q,o,\ell)\in\mathcal{C}^{{\rm in}}}\psi_{i_n^{q}}(y_{n})&\pi_{k(n),i_n^q}^{\ell}(\overline{a})\gamma_{k(n),i_n^q}^{o,\ell}(\overline{a})
V_{n}(t_{k(n)+1},y_{k(n),i_n^q}^{o,\ell}(\overline{a}))+\sum_{q=1}^{M}\psi_{i_n^{q}}(y_{n})\tau_{k(n),i_n^q}(\overline{a})f(t_{k(n)},x_{i_n^q},\overline{a}).\nonumber
\end{align}
By compactness, there exist $\{\overline{\pi}_{q}^{\ell}\,|\,q\in\{1,\hdots,M\},\,\ell\in\I\}\subset [0,1]$ and $\{\overline{\gamma}_{q}^{o,\ell}\,|\,q\in\{1,\hdots,M\},(o,\ell)\in\{+,-\}\times\I\}\subset [0,1]$ such that  $\sum_{\ell\in\I}\overline{\pi}^{\ell}_{q}=1$ for all $q\in\{1,\hdots,M\}$,  $\overline{\gamma}_{q}^{+,\ell}+\overline{\gamma}_{q}^{-,\ell}=1$  for all $(q,\ell)\in\{1,\hdots,M\}\times\I$, and, up to a subsequence,  $\pi_{k(n),i_{n}^{q}}^{\ell}(\overline{a})\to\overline{\pi}_{q}^{\ell}$, and $\gamma_{k(n),i_{n}^{q}}^{o,\ell}(\overline{a})\to \overline{\gamma}_{q}^{o,\ell}$. Since $V_{n}(s_n,y_n)\to\overline{V}(t,x)$ and $\tau_{k(n),i_{n}^{q}}(\overline{a})\to 0$ for all $q\in\{1,\hdots,M\}$, it follows from~\eqref{eq:v_n_convex_prelimit}, the continuity of $\Psi$,  and~\eqref{discrete_semi_limits} that 
\begin{equation}
\label{eq:v_n_convex_limit}
\overline{V}(t,x)\leq \left(\sum_{(q,o,\ell)\in\mathcal{C}^{{\rm out}}}\overline{\psi}_{q}\overline{\pi}_{q}^{\ell}\overline{\gamma}_{q}^{o,\ell}\right)\Psi(t,x)+\left(\sum_{(q,o,\ell)\in\mathcal{C}^{{\rm in}}}\overline{\psi}_{q}\overline{\pi}_{q}^{\ell}\overline{\gamma}_{q}^{o,\ell}\right)\overline{V}(t,x).
\end{equation}
Using that 
$$
\sum_{(q,o,\ell)\in\mathcal{C}^{{\rm out}}}\overline{\psi}_{q}\overline{\pi}_{q}^{\ell}\overline{\gamma}_{q}^{o,\ell}+\sum_{(q,o,\ell)\in\mathcal{C}^{{\rm in}}}\overline{\psi}_{q}\overline{\pi}_{q}^{\ell}\overline{\gamma}_{q}^{o,\ell}=1,$$
we deduce that 
\begin{equation}
\label{eq:inequality_V_Psi_weights}
\left(\sum_{(q,o,\ell)\in\mathcal{C}^{{\rm out}},\, q\in\I_{+}}\overline{\psi}_{q}\overline{\pi}_{q}^{\ell}\overline{\gamma}_{q}^{o,\ell}\right)\overline{V}(t,x)\leq \left(\sum_{(q,o,\ell)\in\mathcal{C}^{{\rm out}},\, q\in\I_{+}}\overline{\psi}_{q}\overline{\pi}_{q}^{\ell}\overline{\gamma}_{q}^{o,\ell}\right)\Psi(t,x). 
\end{equation}
Notice that, for every $q\in\I_{+}$, if $\ell\in\I$ and $\tilde{\ell}\in\I$ are such that $(+,\ell)\in I_{q}^{{\rm out}}$ or $(-,\ell)\in I_{q}^{{\rm out}}$ and $(+,\tilde{\ell})\notin I_{q}^{{\rm out}}$ and $(-,\tilde{\ell})\notin I_{q}^{{\rm out}}$, by~\eqref{weights_multi_dimensional_case} we have  $\tau_{k(n),i_n^q}^{\ell}(\overline{a})<\Delta t_n$ and $\tau_{k(n),i_n^q}^{\tilde{\ell}}(\overline{a})=\Delta t_n$. Therefore,~\eqref{eq:tau} implies that, for every $n\in\NN$, $\pi_{k(n),i_{n}^{q}}^{\ell}(\overline{a})\geq \pi_{k(n),i_{n}^{q}}^{\tilde{\ell}}(\overline{a})$ and, hence, $\overline{\pi}_{q}^{\ell}\geq \overline{\pi}_{q}^{\tilde{\ell}}$. It follows that for every $q\in\I_{+}$, there exists $(\overline{o},\overline{\ell})\in\{+,-\}\times\I$ such that $(q,\overline{o},\overline{\ell})\in \mathcal{C}^{{\rm out}}$ and $\overline{\pi}_{q}^{\overline{\ell}}\geq 1/p$. Suppose, without loss of generality, that $\overline{o}=+$. Then,  if $(-,\overline{\ell})\notin I_{q}^{{\rm out}}$, it follows from~\eqref{weights_multi_dimensional_case} that   $\gamma_{k(n),i_{n}^{q}}^{+,\overline{\ell}}\geq 1/2$ and hence $\overline{\gamma}_{q}^{+,\overline{\ell}}\geq 1/2$. On the other hand, if  $(-,\overline{\ell})\in I_{q}^{{\rm out}}$, using that $\gamma_{k(n),i_{n}^{q}}^{+,\overline{\ell}}+\gamma_{k(n),i_{n}^{q}}^{-,\overline{\ell}}=1$, we have that $\max\{\gamma_{k(n),i_{n}^{q}}^{+,\overline{\ell}},\gamma_{k(n),i_{n}^{q}}^{-,\overline{\ell}}\}\geq 1/2$ and hence $\max\{\overline{\gamma}_{q}^{+,\overline{\ell}},\overline{\gamma}_{q}^{-,\overline{\ell}}\}\geq 1/2$. Thus, there exists $\tilde{o}\in\{+,-\}$ such that $(q,\tilde{o},\overline{\ell})\in \mathcal{C}^{{\rm out}}$ and $\overline{\gamma}_{q}^{\tilde{o},\overline{\ell}}\geq 1/2$. In turn, 
$$
\sum_{(q,o,\ell)\in\mathcal{C}^{{\rm out}},\, q\in\I_{+}}\overline{\psi}_{q}\overline{\pi}_{q}^{\ell}\overline{\gamma}_{q}^{o,\ell}\geq \frac{\sum_{\ell\in\I_{+}}\overline{\psi}_q}{2p}=\frac{1}{2p}>0
$$
and, by~\eqref{eq:inequality_V_Psi_weights}, we obtain that $\overline{V}(t,x)\leq \Psi(t,x)$, which is a contradiction. 

Thus, let us assume that there exists $\overline{q}\in\mathcal{I}_{+}$ such that $\lambda_{k(n),i^{\overline{q}}_{n}}^{\pm,\ell}(\overline{a})=1$ for all $\ell\in\I$ and $n\in\NN$ large enough. Without loss of generality, suppose that $\overline{q}=1$. Then it follows from~\eqref{eq:extension},~\eqref{nonlinear_fully}, ~\eqref{eq:dis}, and Proposition~\ref{prop:properties}\eqref{prop:monotonia}\&\eqref{prop:constant_addition} that 
\begin{align*}
V_{n} (s_n,y_n)&=\sum_{q=1}^{M}\psi_{i_{n}^{q}}(y_n)(V_n)_{k(n),i_{n}^{q}}\leq \psi_{i_{n}^{1}}(y_n)\mathcal S^{\rm fd}_{k(n),i_n^1}( V_n(t_{k(n)+1},\cdot),\overline{a})+\sum_{q=2}^{M}\psi_{i_{n}^{q}}(y_n)(V_n)_{k(n),i_{n}^{q}}\\[6pt]
&\leq \psi_{i_{n}^{1}}(y_n)\mathcal S^{\rm fd}_{k(n),i_n^1}(\phi(t_{k(n)+1},\cdot)+\xi_{n},\overline{a})+\sum_{q=2}^{M}\psi_{i_{n}^{q}}(y_n)(\phi(t_{k(n)},x_{i_n^q})+\xi_{n})\\
&= \psi_{i_{n}^{1}}(y_n)\mathcal S^{\rm fd}_{k(n),i_n^1}(\phi(t_{k(n)+1},\cdot),\overline{a})+\sum_{q=2}^{M}\psi_{i_{n}^{q}}(y_n)\phi(t_{k(n)},x_{i_n^q})+\xi_{n},
\end{align*}
which, by definition of $\xi_n$, yields 
\begin{equation}
\label{eq:inequality_scheme_phi_0}
\phi(s_n,y_n)\leq \psi_{i_{n}^{1}}(y_n)\mathcal S^{\rm fd}_{k(n),i_n^1}(\phi(t_{k(n)+1},\cdot),\overline{a})+\sum_{q=2}^{M}\psi_{i_{n}^{q}}(y_n)\phi(t_{k(n)},x_{i_n^q}).
\end{equation}
On the other hand, since $V_{n}(\cdot,y_n)$ is constant in $[t_{k(n)},t_{k(n)+1})$ and $s_{n}$ maximizes $V_{n}(\cdot,y_n)-\phi(\cdot,y_n)$, we have that either $s_n=t_{k(n)}$ or $\partial_{t}\phi(s_n,y_n)=0$. In both cases, it holds that $\phi(s_n,y_n)=\phi(t_{k(n)},y_n)+O((\Delta t_n)^{2})$. Thus, by~\eqref{interp_estim} and~\eqref{eq:inequality_scheme_phi_0}, we have that
$$
\sum_{q=1}^{M} \psi_{i_{n}^{q}}(y_n)\phi(t_{k(n)},x_{i_n^q})\leq \psi_{i_{n}^{1}}(y_n)\mathcal S^{\rm fd}_{k(n),i_n^1}(\phi(t_{k(n)+1},\cdot),\overline{a})+\sum_{q=2}^{M}\psi_{i_{n}^{q}}(y_n)\phi(t_{k(n)},x_{i_n^q})+O((\Delta t_n)^{2}+(\Delta x_n)^{2})
$$
and hence, for all $n\in\NN$ large enough, we have 
$$
\phi(t_{k(n)},x_{i_n^1})\leq \mathcal S^{\rm fd}_{k(n),i_n^1}(\phi(t_{k(n)+1},\cdot),\overline{a})+O((\Delta t_n)^{2}+(\Delta x_n)^{2}).
$$
In turn, the consistency of the scheme, which holds under our conditions over $\Delta t_n$ and $\Delta x_n$, yields  
$$
-\partial_{t}\phi(t,x)+\sup_{a\in A}\mathcal{L}^{a}(t,x,D_x\phi(t,x),D_x^2\phi(t,x))=-\partial_{t}\phi(t,x)+\mathcal{L}^{\overline{a}}(t,x,D_x\phi(t,x),D_x^2\phi(t,x))\leq 0,
$$
which shows that {\rm(i.2)} in Definition~\ref{weak_boundary_conditions} holds. Thus, $\overline{V}$ is a viscosity subsolution to \eqref{eq:HJB}.

Finally, by~\eqref{discrete_semi_limits} and the strong comparison principle, we have $\overline{V}=\underline{V}$ in $[0,T]\times\Omega$, which, by Proposition~\ref{prop:uniqueness_sol_weak_sense}, yields $\overline{V}=\underline{V}=V$ in $[0,T]\times\Omega$, where $V$ is given by~\eqref{value_function_exit_from_open}. Thus, the convergence in~\eqref{eq:uniform_convergence_value_function_weak_framework} follows from~\cite[Lemma 1.9 in Chapter V]{MR1484411}.
\end{proof}

\section{Numerical tests}
\label{sec:tests}
We present some numerical simulations to show the behavior of the proposed scheme in four different problems:
\begin{enumerate}
\item[1.] A one-dimensional problem, in which we consider a vanishing diffusion term;
\item[2.] A two-dimensional problem, with degenerate diffusion and a non-homogeneous Dirichlet boundary condition, for which a classical solution exists;
\item[3.] A two-dimensional problem with no diffusion and an  homogeneous Dirichlet boundary condition satisfied in the weak sense;
\item[4.] A stochastic minimum  time problem, modeling the escape from a room with two exits.
\end{enumerate}

In the second and third problems, the analytic solution $V$ is known which allows to analyze the numerical convergence. Let us define the error in the $L^\infty$-discrete norm at time $t=0$ as
$$ E^{\infty}_{\Delta t,\Delta x}=\max_{x_i\in\mathcal{G}_{\Delta x}}|V_{0,i}-V(0, x_i)|,$$
where $\{V_{0,i}\,|\,i\in\I_{\Delta x}\}$ are computed with~\eqref{nonlinear_fully}. The numerical convergence order is given by
$$p^{\infty}_{\Delta t,\Delta x}=\log_2 \left(E^{\infty}_{\Delta t,\Delta x}/E^{\infty}_{2\Delta t,2\Delta x}\right).$$

In all tests, we use an hyperbolic CFL condition $\Delta t=O(\Delta x)$.  In Tests 3 and 4, defined on rectangular domains, we made use of  structured grids. 

\subsection{Test 1: One-dimensional problem with a varying diffusion term}\label{sec:1d}
 We consider the simple Example \ref{ex1}, with the addition of a second-order term,
\begin{align*}
-\partial_t v- \nu v_{xx}+|v_x|&=0\quad\text{for }   (t,x)\in[0,1)\times(0,1),\\
 v(1,x)&=0\quad\text{for }  x\in[0,1],\\
v(t,0)&=1-t\quad v(t,1) =0\quad\text{for }t\in[0,1),
\end{align*}
where $\nu\geq 0$. Figure~\ref{fig:test0} shows the numerical solution $V_{\Delta t,\Delta x}(0,\cdot)$ computed with  $\Delta t = \Delta x = 0.01$
and values  $\nu=1,\,0.1,\,0.01,\,0$  for the viscosity parameter.
We observe that the numerical solution does not develop spurious oscillations near $x=0$, where a boundary layer appears for very small values of $\nu$. In the case $\nu=0$, $V_{\Delta t,\Delta x}(0,\cdot)$ shows a discontinuity at $x=0$, as expected from the explicit expression~\eqref{eq:explicit_V_example}. 
\begin{figure}[htbp]
\centering
\includegraphics[width=0.4\textwidth]{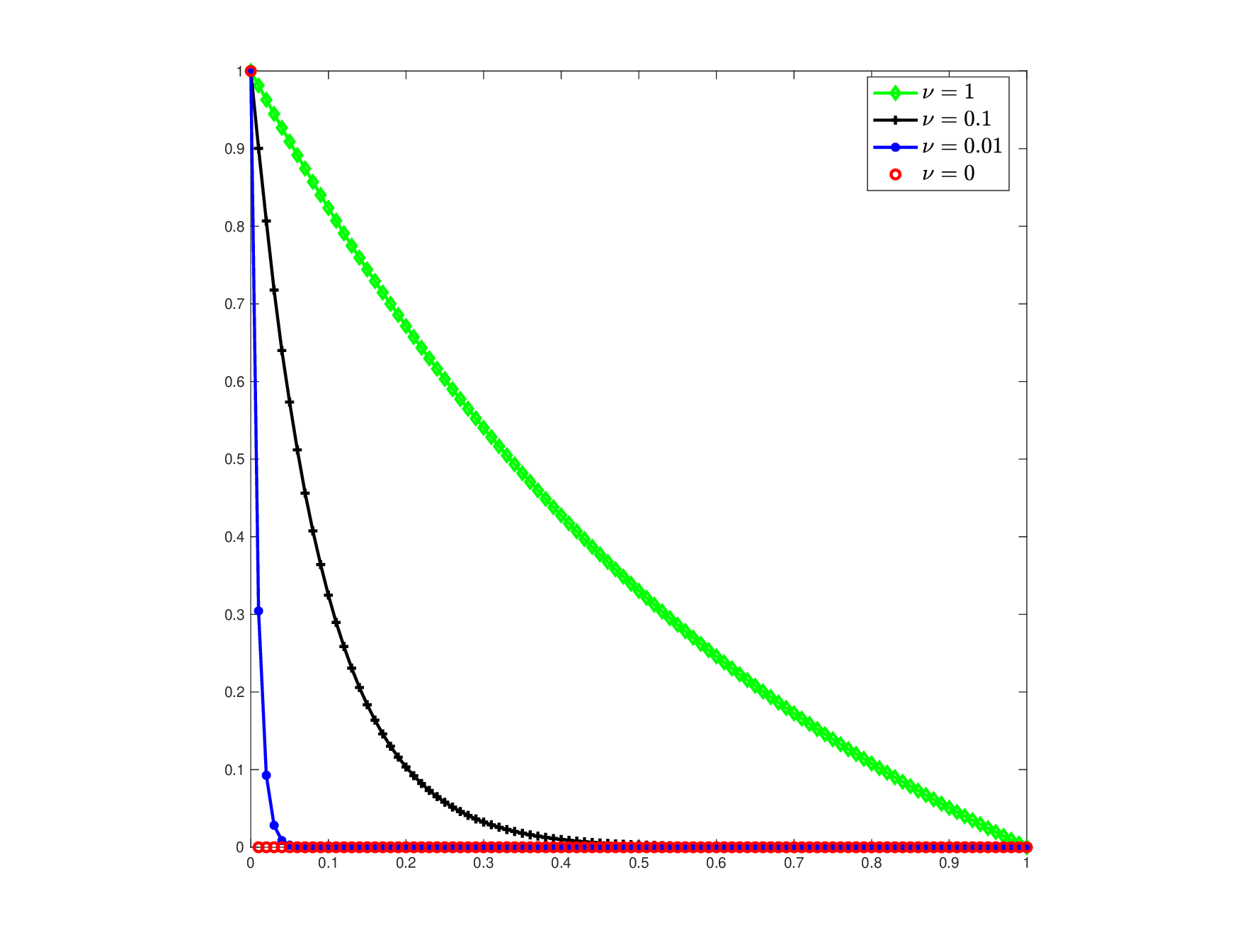}  
\caption{Numerical solutions $V_{\Delta t,\Delta x}(0,\cdot)$ of Test 1 computed with $\nu=1,\,0.1,\,0.01,\,0$. The abscissa and the ordinate represent, respectively, the space variable and the approximated value function at the initial time.}
\label{fig:test0}
\end{figure}

\subsection{Test 2: Two-dimensional problem on a circular domain with a smooth solution.}
We study a benchmark problem, defined on a circular domain with non-homogeneous Dirichlet boundary conditions, which admits a classical and explicit solution. Similar problems have been considered in~\cite{BoZi04} and~\cite{ReisRota17} with periodic and homogeneous Dirichlet boundary conditions, respectively. 
More precisely, let us consider 
the equation
\begin{align*}
-\partial_t v- \frac{1}{2}\text{Tr}[\sigma \sigma^\top D^{2}_{x}v]+|D_xv|&=f\quad\text{in }Q_{T},\\
 v &=\Psi \quad\text{on } \partial^{*}Q_{T},
\end{align*}
where $\Omega=\{x=(x_1,x_2)\in\R^2\,|\, x_1^2+x_2^2< 1\}$, $T=1$,
\begin{multline*}
f(t,x)=(t-1/2)\sin(x_1)\sin(x_2)+(t+1/2)\Big(\sqrt{\cos^2(x_1)\sin^2(x_2)+\sin^2(x_1)\cos^2(x_2)}\\
-2\sin(x_1+x_2)\cos(x_1+x_2)\cos(x_1)\cos(x_2)\Big),
\end{multline*}
$\sigma^1(t,x)=\sqrt{2}\left(\sin(x_1+x_2),\cos(x_1+x_2)\right)^\top$, $\sigma^2(t,x)=(0,0)^\top$, and $\Psi(t,x)=(t+1/2)\sin (x_1)\sin(x_2)$. This problem has a degenerate diffusion term but admits a smooth solution given by  $V(t,x)= (t+1/2)\sin (x_1)\sin(x_2)$ in $\overline{Q}_{T}$. 

In Figure~\ref{fig:test2}, we display the numerical solution $V_{\Delta t,\Delta x}(0,\cdot)$ computed over an unstructured triangular mesh with maximum mesh size $\Delta x=0.125$ and time step $\Delta t= \Delta x/2$.
In Table~\ref{tab:test2}, we show the errors and  the orders of convergence for the cases $\Delta t =\Delta x$ and  $\Delta t =\Delta x/2$. We observe an order of convergence near $1$ and lower errors for smaller time steps.

\begin{figure}[ht!]
\includegraphics[width=0.5\textwidth]{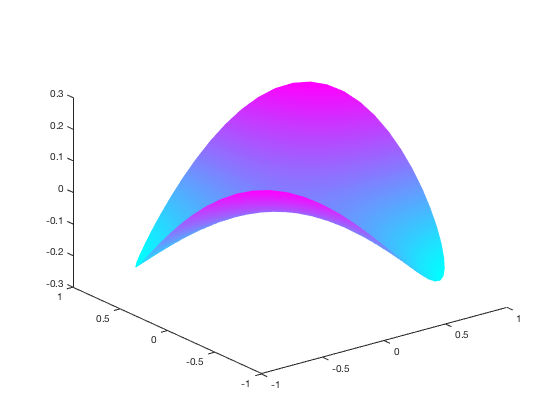}\includegraphics[width=0.5\textwidth]{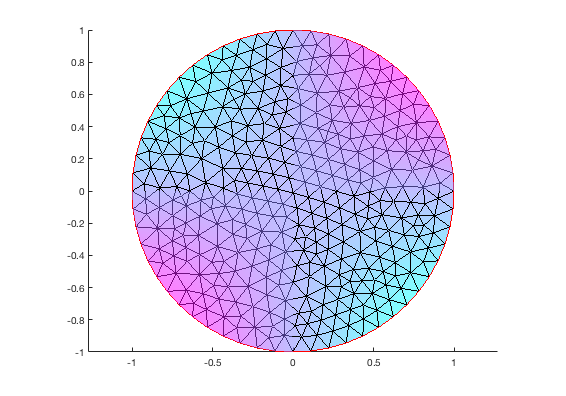} 
\caption{On the left, we display the numerical solution $V_{\Delta t,\Delta x}(0,\cdot)$ of Test 2 computed with ${\Delta x}=0.125$ and $\Delta t=\Delta x/2$. The $x_1x_2$ plane and the $x_3$-axis represent, respectively, the space variable and the approximated value function at the initial time. On the right, we present the projection of the numerical solution onto the $x_1x_2$ plane together with
the computational mesh.}
\label{fig:test2}
\end{figure}

\begin{table}[h]
\centering
\begin{tabular}{|c|c|c|c|c|c|}
\cline{3-6}
\multicolumn{2}{c|}{} &$E^\infty_{\Delta t,\Delta x}$  &$p^{\infty}_{\Delta t,\Delta x}$&$E^\infty_{\Delta t,\Delta x}$   &$p^\infty_{\Delta t,\Delta x}$\\
\hline
${\Delta x}$ &$N_{\Delta x}$  &\multicolumn{2}{c|}{ $\Delta t= {\Delta x}$}&\multicolumn{2}{c|}{$\Delta t= {\Delta x}/2$}\\
 \hline
$5.00\cdot 10^{-1}$     & 48   &$1.26\cdot 10^{-1}$  &$-$      &$4.38\cdot 10^{-2}$& $-$    \\
$2.50\cdot 10^{-1}$     &172   &$5.02\cdot 10^{-2} $ &$1.32 $  &$2.09\cdot 10^{-2}$&$1.06 $ \\
$1.25\cdot 10^{-1}$     &694  &$2.45\cdot 10^{-2} $  &$1.03$  &$1.14\cdot 10^{-2}$ &$0.87$ \\
$6.25\cdot 10^{-2}$     &2686 &$1.24\cdot 10^{-2} $  &$0.98$  &$6.19\cdot 10^{-3}$ &$0.88$\\
\hline
\end{tabular}
\vspace{0.2cm}
\caption{Maximum mesh size ${\Delta x}$ (first column), number of vertices $N_{\Delta x}$ (second column),
$L^\infty$-errors $E^\infty_{\Delta t,\Delta x}$ and orders of convergence $p^{\infty}_{\Delta t,\Delta x}$ for $\Delta t = {\Delta x}$ (third and fourth columns) and  for $\Delta t =2 {\Delta x}$ (last two columns) in Test 2.}
\label{tab:test2}
\end{table}

\subsection{Test 3: Two-dimensional problem on a square domain with a nonsmooth solution}
\label{test3}
We consider here the following first-order example, discussed in~\cite{W08,KY2015},
\begin{align*}
-\partial_t v+ \underset{a\in A}\max\left\{
-a^{\top}D_x v-\frac{1}{4}a_1^2-a_2^2\right\} 
&=1\quad\text{in }Q_{T},\\
 v &=\Psi \quad\text{on } \partial^{*}Q_{T},
\end{align*}
where $\Omega=(0,1)\times(0,1)$, $A=[-2,2]\times [-2,2]$, $T=1.5$,  and, setting $\psi(x)=\min(x_1,1-x_1,2x_2,2(1-x_2))$,
$$
\Psi(t,x)=
\begin{cases}
0&\text{if }(t,x)\in [0,T)\times\partial\Omega,\\
-2\psi(x)&\text{if }(t,x)\in\{T\}\times\overline{\Omega}.
\end{cases}
$$

Notice that at the final time $T$ the solution satisfies the boundary condition in the strong sense. In Figure \ref{fig:test1}, we present the numerical solution at times $t=1.5,1.2,0.9,0.3,0.1,0$ in the open space domain $\Omega$, calculated using a uniform structured grid with $\Dx=\Dt=0.005$. We observe, as in~\cite{W08,KY2015}, the existence of $\bar{t}\in (0,T)$ such that the boundary condition holds in the weak sense for times $t\in(\bar{t},T)$ and in the strong sense for $t\in [0,\bar{t}\,]$. Moreover, as $t\to 0$, we notice that the solution converges towards an approximation of the non-smooth function $\psi$. In Table \ref{tab:test3}, we show the errors obtained by comparing the numerical solution at time $0$ with $\psi$, and the convergence order for cases $\Delta t = {\Delta x}$ and $\Delta t =2 {\Delta x}$. The table shows a convergence rate of order 1, with smaller errors for smaller time steps.
Note that the absence of a second-order term in the equation allows one to take time steps larger than space steps.
\begin{table}[h]
\centering
\begin{tabular}{|c|c|c|c|c|c|}
\cline{3-6}
\multicolumn{2}{c|}{} &$E^\infty_{\Delta t,\Delta x}$  &$p^{\infty}_{\Delta t,\Delta x}$&$E^\infty_{\Delta t,\Delta x}$   &$p^\infty_{\Delta t,\Delta x}$\\
\hline
${\Delta x}$ &$N_{\Delta x}$  &\multicolumn{2}{c|}{ $\Delta t= {\Delta x}$}&\multicolumn{2}{c|}{$\Delta t= {2 \Delta x}$}\\
 \hline
$4.00\cdot 10^{-2}$  &$\;\;2500$      &$2.08\cdot 10^{-4}$  &  $-$   &$6.25\cdot 10^{-4}$& $-$      \\
$2.00\cdot 10^{-2}$ &$\;\;5000$      &$1.02\cdot 10^{-4} $ &$0.97 $  &$3.06\cdot 10^{-4}$&$1.05 $ \\
$1.00\cdot 10^{-2}$ &$10000$     &$5.05\cdot 10^{-5} $  &$1.01$  &$1.51\cdot 10^{-4} $  &$0.98$ \\
$5.00\cdot 10^{-3}$ &$20000$     &$2.51\cdot 10^{-5} $  &$1.00$  &$7.53\cdot 10^{-5}$ &$1.00$\\
\hline
\end{tabular}
\vspace{0.2cm}
\caption{
Maximum mesh size ${\Delta x}$ (first column), number of vertices $N_{\Delta x}$ (second column),
$L^\infty$-errors $E^\infty_{\Delta t,\Delta x}$ and orders of convergence $p^{\infty}_{\Delta t,\Delta x}$ for $\Delta t = {\Delta x}$ (third and fourth columns) and  for $\Delta t =2 {\Delta x}$ (last two columns) in Test 3.}\label{tab:test3}
\end{table}

\begin{figure}[ht!]
\centering
    \begin{subfigure}{0.3\textwidth}
        \centering
        \includegraphics[width=\textwidth]{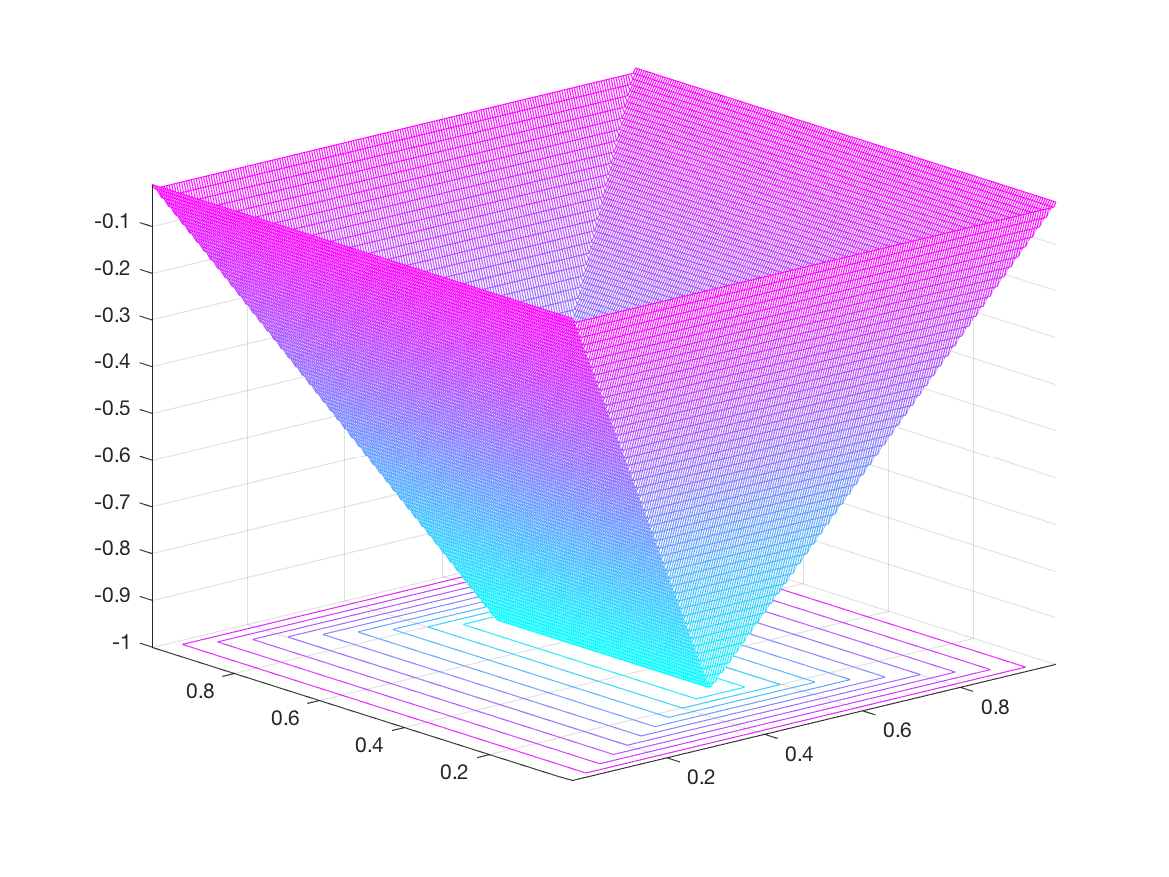}
        \caption{$t = 1.5$}
        \label{fig:test1_t0}
    \end{subfigure}
    \begin{subfigure}{0.3\textwidth}
        \centering
        \includegraphics[width=\textwidth]{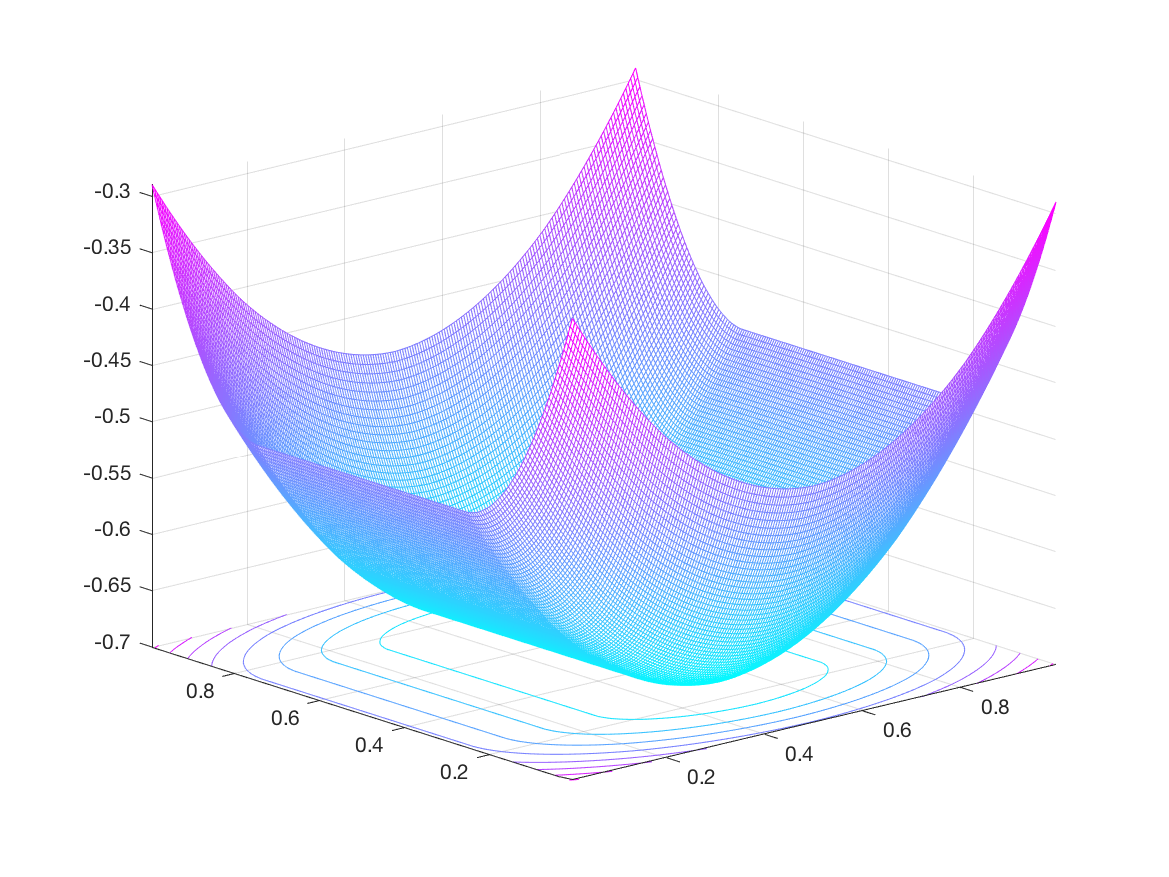}
        \caption{$t = 1.2$}
        \label{fig:test1_t0.3}
    \end{subfigure}
    \begin{subfigure}{0.3\textwidth}
        \centering
        \includegraphics[width=\textwidth]{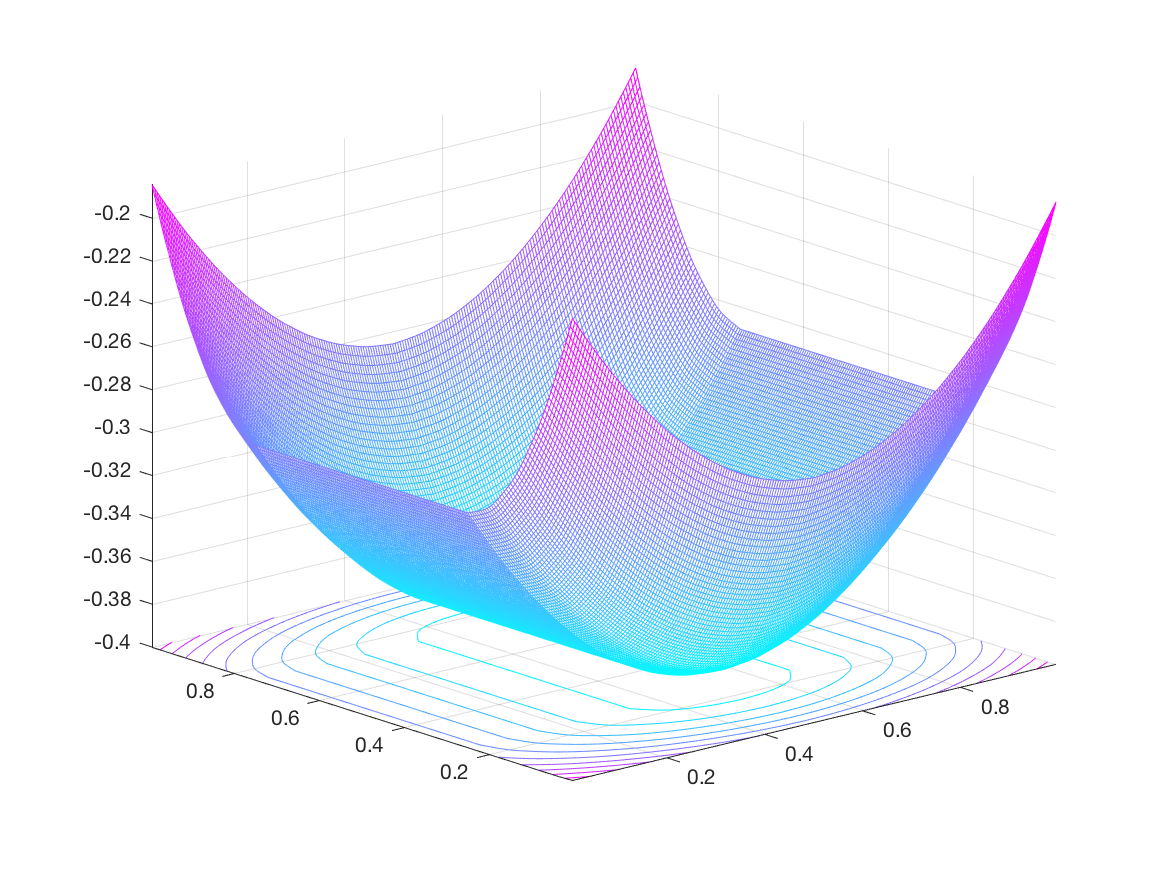}
        \caption{$t = 0.9$}
        \label{fig:test1_t0.6}
    \end{subfigure}
    
    \begin{subfigure}{0.3\textwidth}
        \centering
        \includegraphics[width=\textwidth]{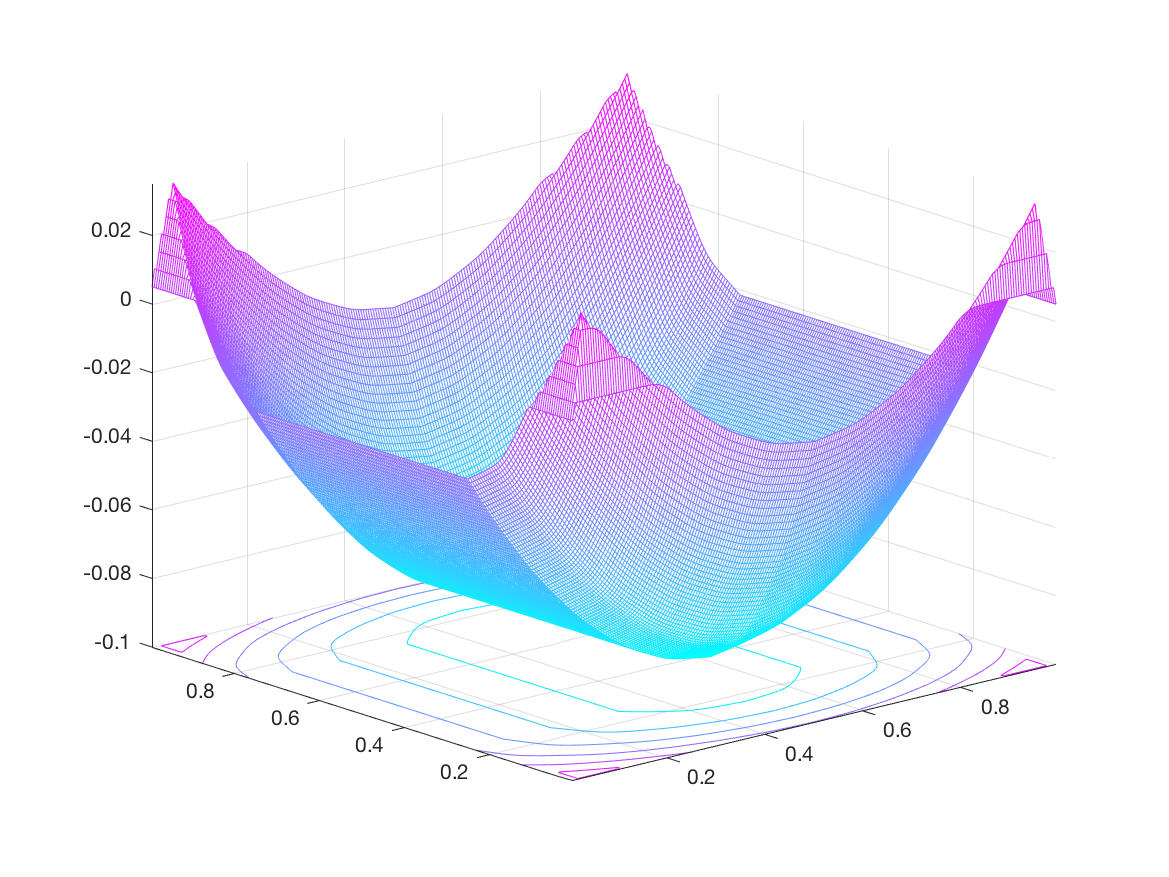}
        \caption{$t = 0.3$}
        \label{fig:test1_t0.9}
    \end{subfigure}

    \begin{subfigure}{0.3\textwidth}
        \centering
        \includegraphics[width=\textwidth]{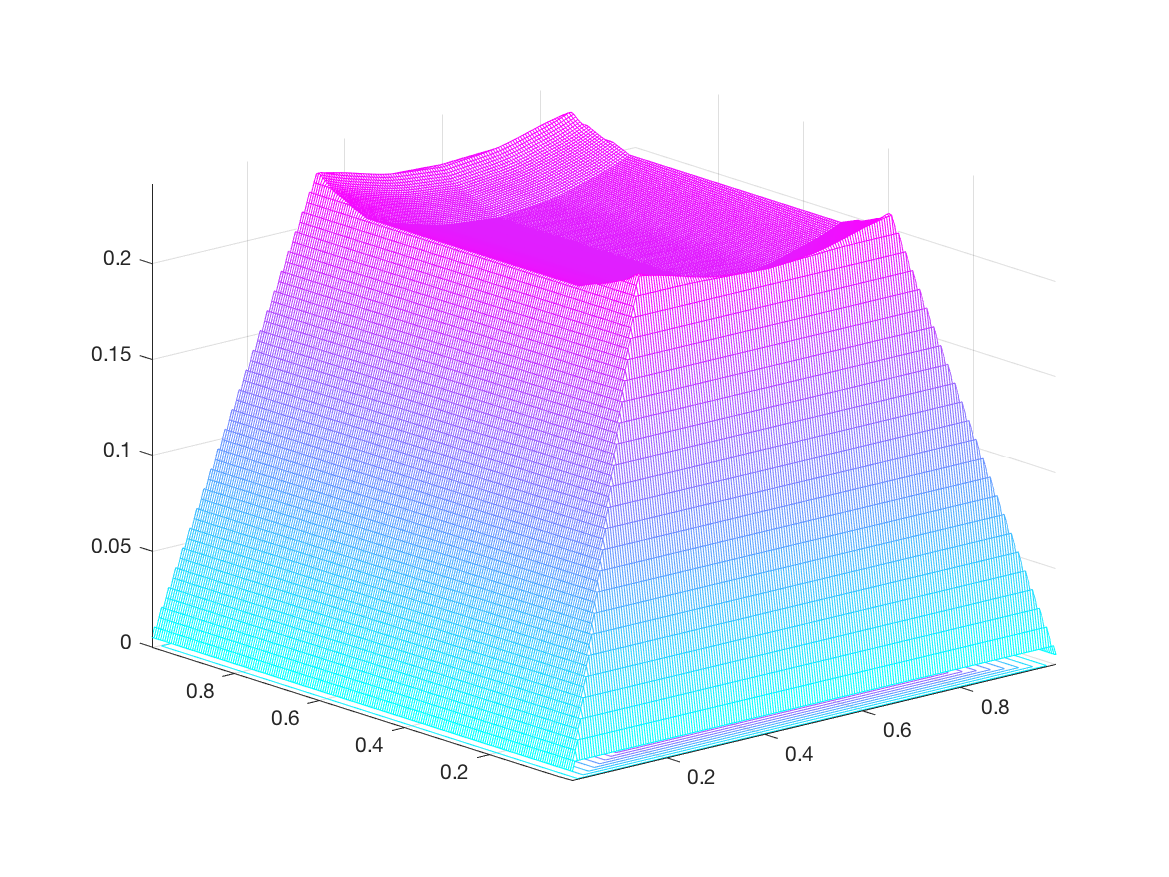}
        \caption{$t = 0.1$}
        \label{fig:test1_t1.5}
        \end{subfigure}
        \begin{subfigure}{0.3\textwidth}
        \centering
        \includegraphics[width=\textwidth]{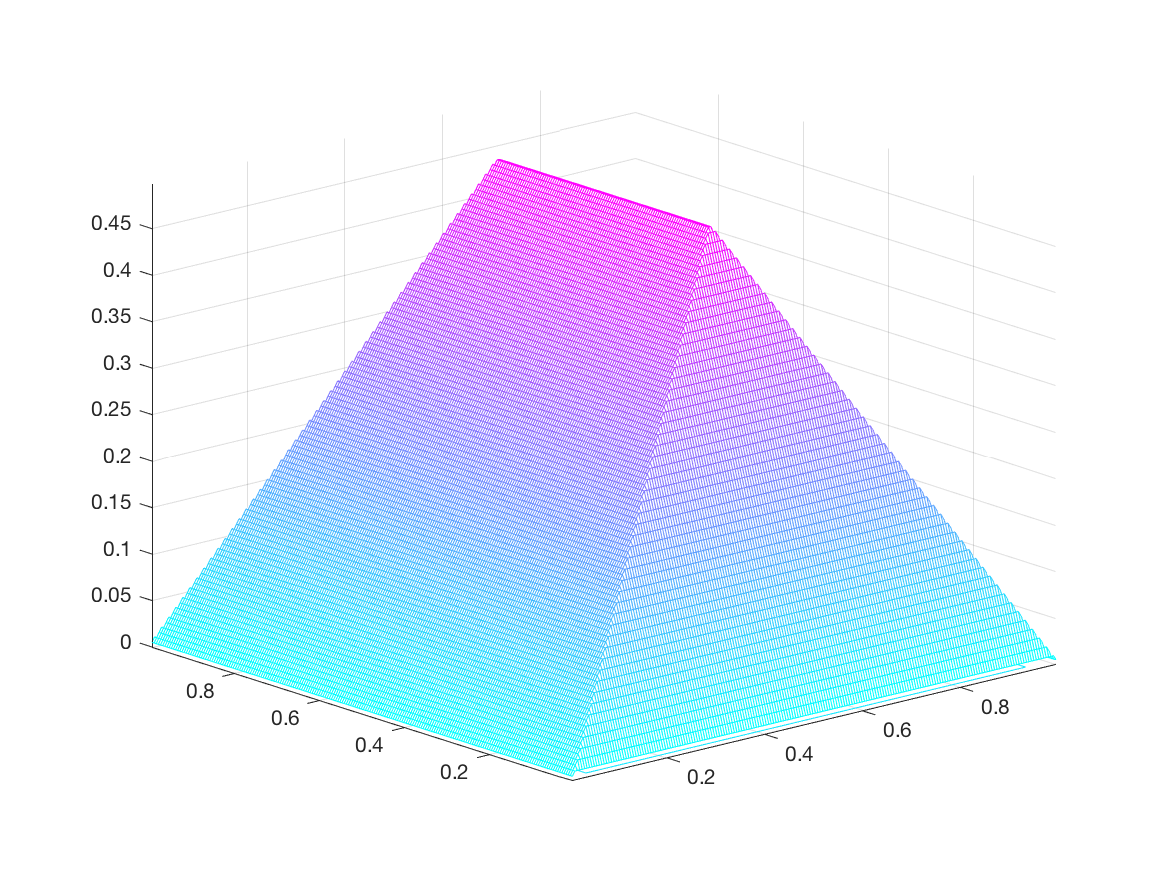}
        \caption{$t = 0$}
        \label{fig:test1_t1.2}
    \end{subfigure}
\caption{Numerical solution $V_{\Dt,\Dx}$ of Test 3 at different times. (A) $t = 1.5$, (B) $t = 1.2$, (C) $t = 0.9$, (D) $t = 0.3$, (E) $t = 0.1$, (F) $t = 0$}
\label{fig:test1}
\end{figure}
\subsection{Test 4: Stochastic minimum time problem of exiting from a room with two doors}
 We consider problem~\eqref{eq:HJB} to model the exit from a bounded rectangular domain $\Omega = (-1,1) \times (-0.5,0.5)$ in a time horizon $T = 5$. We take $f(t,x,a) = \frac{1}{2}|a|^2+40$, $b(t,x,a) = a$, $A=\{a\in\R^2\,|\, |a|\leq1\}$, and $\Psi(T,x) = 0$ for $x \in\overline{\Omega}$.
We introduce a degenerate diffusion $\sigma(t,x,a)=\left(\sigma^1(x)\;\;\sigma^2\right)$, where $\sigma^1(x)=s(x)(1\;0)^{\top}$, with $s(x)=10\max(-x_1^2-x_2^2+(0.4)^2,0)$, and $\sigma^2=(0\;0)^{\top}$. 
We prescribe Dirichlet boundary conditions on $\Gamma_1 = \{x \in \partial \Omega\,|\, x_1 = -1, |x_2| \leq 0.2\}$ and $\Gamma_2 = \{x \in \partial \Omega\,|\, x_1 = 1, |x_2| \leq 0.2\}$, representing two exits, and on the remaining part of the boundary $\Gamma_w = \partial \Omega \setminus (\Gamma_1 \cup \Gamma_2)$, which models the walls of the room.
The cost of exiting through the two doors is set as $\Psi(t,x) = 0$ if $x \in \Gamma_1$ and $\Psi(t,x) = 0.2$ if $x \in \Gamma_2$. To model the wall, we prescribe a high exit cost on $\Gamma_{w}$. In our test, we have set $\Psi(t,x) = 100$ for $x \in \Gamma_w$. Notice that $\Psi$ is discontinuous and hence our convergence results do not apply. Despite this, we observe below that our scheme still provides compelling numerical results.

In Figure~\ref{fig:test3a}, we present the approximation of the viscosity solution computed on a uniform structured grid with $\Delta t = \Delta x = \sqrt{2}/50$. The solution is shown at time \(t = 0\) (left), along with its contour lines on the \(x_1x_2\)-plane (right). In the right panel, we also show six trajectories starting from the points $P_{1}=(-0.1, -0.3)$, $P_{2}=(-0.1,0.1)$, $P_{3}=(-0.1,0.3)$, $P_{4}=(0.2,-0.3)$, $P_{5}=(0.3,0.2)$, and $P_{6}=(0.2, 0.3)$. 
Provided that $v$ is regular enough, standard verification results (see, e.g.,~\cite[Section III.3]{FS06}) imply that, starting from $P_{i}$ and before hitting the boundary, an optimal stochastic trajectory solves 
$$
\begin{aligned}
\dd Y(s)&=\text{Proj}_{A}(-D_{x}v(s,Y(s)))\dd s+\sigma(s,Y(s),\alpha(s))\dd W(s),\\
Y(0)&=P_{i},
\end{aligned}
$$
where $\text{Proj}_{A}(z)=z$, if $z\in A$, and $\text{Proj}_{A}(z)=z/|z|$, otherwise. We approximate the solution to the SDE 
by using a stochastic Euler scheme (see, e.g.,~\cite{MR1214374}) with a time step $\Delta t/2$, where
$D_{x}v$ is approximated by using finite differences applied to the numerical value function $V_{\Delta t, \Delta x}$.

In Figure~\ref{fig:test3b}, we show two additional simulations of six approximate optimal stochastic trajectories, starting from the same initial conditions  $P_{1},\dots ,P_6$. 
We observe that the degenerate diffusion causes oscillations in the horizontal direction of the trajectories and that different choices can be made for exiting; for instance, two out of three trajectories starting from $P_1$ exit through the nearest door.

In Figure~\ref{fig:test3c}, we present the numerical results at time $t=0$ for the deterministic case ($\sigma(t,x,a)\equiv 0$) calculated in the same domain as in the previous simulation, using the same space mesh, and taking $\Delta t=2 \Delta x$. We also show the contour lines on the $x_1x_2$-plane,
along with six approximate optimal trajectories starting again from the points $P_1,\dots,P_6$.
   \begin{figure}[h!]
\centering
	\includegraphics[width=0.4\textwidth]{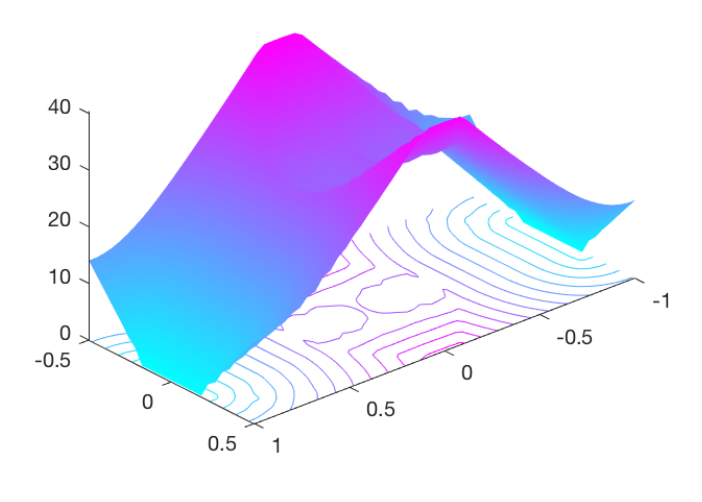}\includegraphics[width=0.4\textwidth]{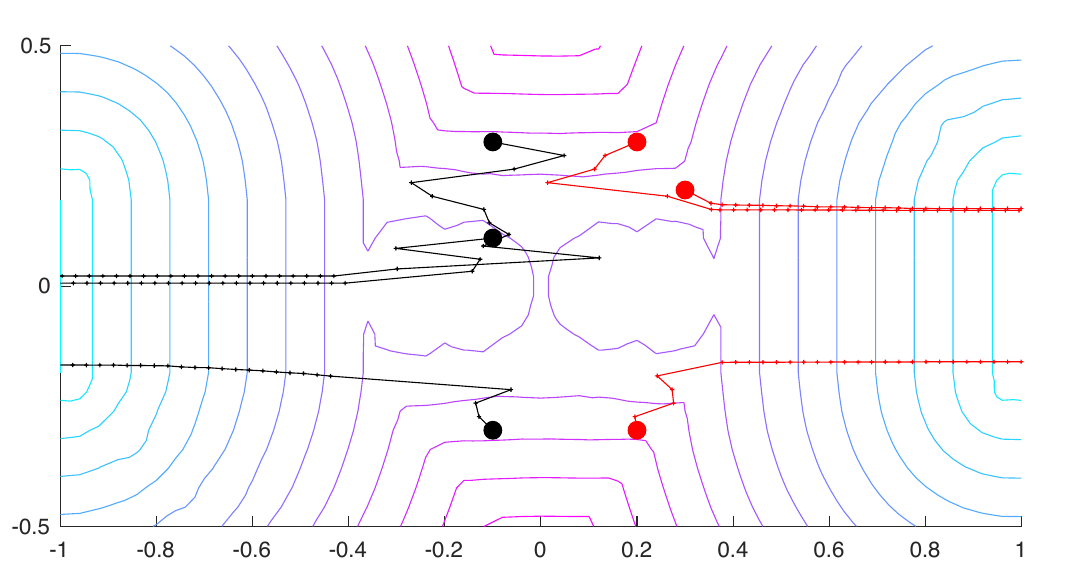} 

\caption{Test 4: approximate viscosity solution at time $t=0$ with $\Delta t = \Delta x = \sqrt{2}/50$ (left). Contour lines on the $x_1x_2$-plane and simulations of approximate optimal trajectories (right). }\label{fig:test3a}

\end{figure}

\begin{figure}[hb!]
\centering
	\includegraphics[width=0.4\textwidth]{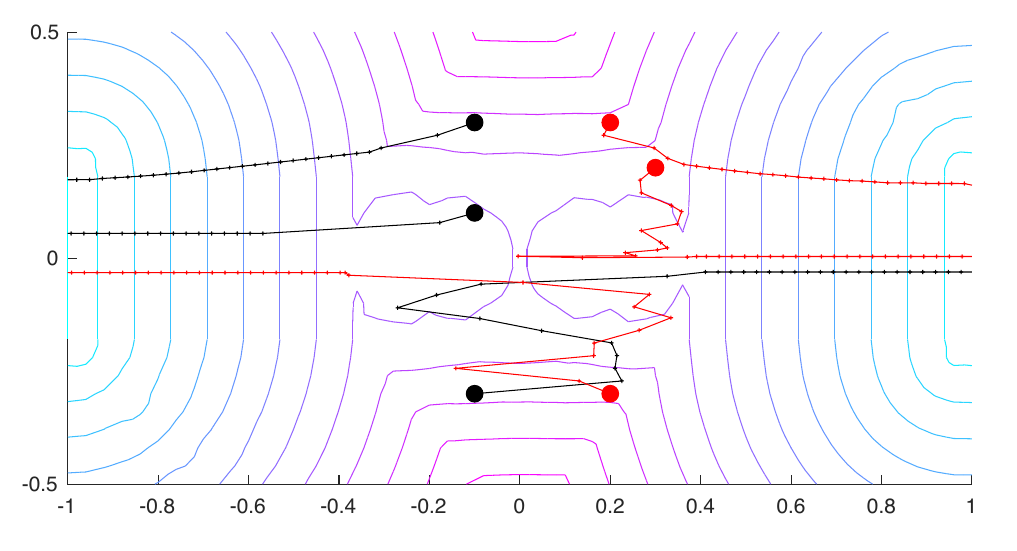}\includegraphics[width=0.4\textwidth]{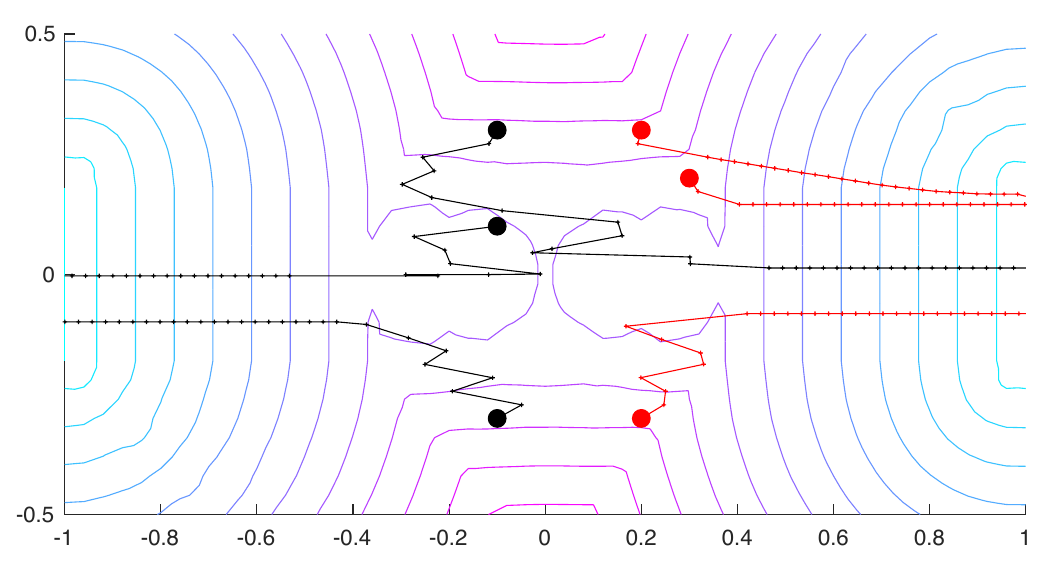} 
	\caption{Test 4: contour lines on the $x_1x_2$-plane and simulations of approximate optimal trajectories. }\label{fig:test3b}

\end{figure}

\begin{figure}[h]
\centering
\includegraphics[width=0.4\textwidth]{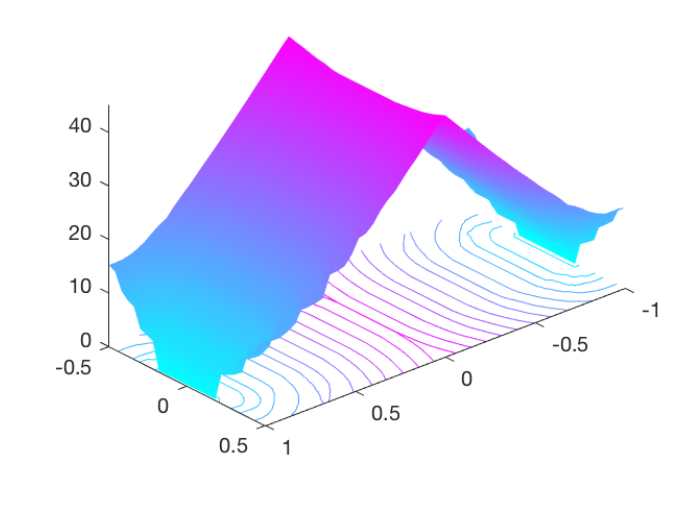}\includegraphics[width=0.4\textwidth]{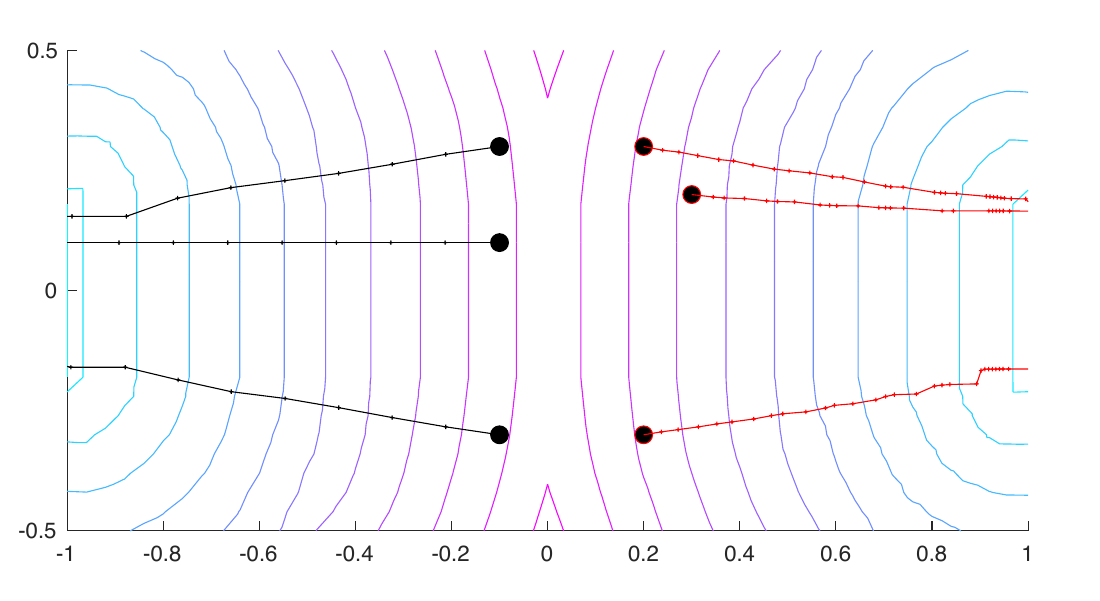} 

\caption{Test 4: approximate viscosity solution at time $t=0$ computed with $\sigma(t,x,a)\equiv0$ and $\Delta t = 2{\Delta x} = \sqrt{2}/25$ (left). Contour lines on the $x_1x_2$-plane  and approximate deterministic optimal trajectories (right). }
\label{fig:test3c}
\end{figure}

\newpage

\appendix
\section{}
\label{Taylor_dim_1_dem}
\begin{proof}[Proof of Lemma~\ref{Taylor_dim_1}]
Let $k\in\I_{\Delta t}^{*}$, $x\in\Omega$, and $\ell\in \mathcal{I}$. We have that
\begin{equation*}
\begin{array}{rcl}
\varphi(t_{k}+\lambda^{\pm}_k(x)\Dt,y^{\pm,\ell}_{k} (x)) &=&  {\varphi}(t_{k},x) + \partial_{t}\varphi(t_k,x) \lambda^{\pm,\ell}_k(x)\Dt+ D_x\varphi(t_k,x)^{\top}(y^{\pm,\ell}_{k} (x)-x)\\[6pt]
\; & \; & +  \left(\int_{0}^{1}\left[\partial_{t} \varphi\big(t_{k}+ \xi \lambda^{\pm,\ell}_k(x)\Dt, x+\xi (y^{\pm,\ell}_{k} (x)-x)\big)-\partial_{t}\varphi(t_k,x) \right]\dd \xi\right) \lambda^{\pm,\ell}_k(x)\Dt\\[6pt]
\; & \; &+ \left(\int_{0}^{1}\left[D_x\varphi(t_{k}+ \xi \lambda^{\pm,\ell}_k(x)\Dt, x+\xi(y^{\pm,\ell}_{k} (x)-x))- D_x\varphi(t_k,x) \right]\dd \xi\right)\cdot(y^{\pm,\ell}_{k} (x)-x)\\[10pt]
\; & =&  \varphi(t_{k},x) + \partial_{t}\varphi(t_k,x) \lambda^{\pm,\ell}_k(x)\Dt+ D_x\varphi(t_k,x)^{\top} (y^{\pm,\ell}_{k} (x)-x)\\[6pt]
\; & \; &+ \left(\int_{0}^{1}\left[D_x\varphi(t_{k}, x+\xi(y^{\pm,\ell}_{k} (x)-x))- D_x\varphi(t_k,x) \right]\dd \xi\right)^{\top} (y^{\pm,\ell}_{k} (x)-x)\\[8pt]
&+&
\lambda^{\pm,\ell}_k(x)\Dt \epsilon_{k}^{1,\pm,\ell}(\Delta t,x) +\epsilon_{k}^{2,\pm,\ell}(\Delta t,x)^{\top} (y^{\pm,\ell}_{k} (x)-x)
\end{array}
\end{equation*}
Hence, we obtain 
\begin{equation}
\label{sviluppo_dim_1}
\begin{array}{rcl}
\varphi(t_{k}+\lambda^{\pm}_k(x)\Dt,y^{\pm,\ell}_{k} (x))   
\; &= & {\varphi}(t_{k},x) + \partial_{t}\varphi(t_k,x) \lambda^{\pm,\ell}_k(x)\Dt+ D_x\varphi(t_k,x)^{\top}(y^{\pm,\ell}_{k} (x)-x)\\[8pt]
\; &\,  & + \frac{1}{2} (y^{\pm,\ell}_{k} (x)-x)^{\top} \big[D_x^2 \varphi(t_k,x) (y^{\pm,\ell}_{k} (x)-x)\big]  +\lambda^{\pm,\ell}_k(x)\Dt \epsilon_{k}^{1,\pm,\ell}(\Delta t,x) \\[8pt]
\; & \; &+(y^{\pm,\ell}_{k} (x)-x)^{\top} \epsilon_{k}^{2,\pm,\ell}(\Delta t,x)+ (y^{\pm,\ell}_{k} (x)-x)^{\top}\big[\epsilon_{k}^{3,\pm,\ell}(\Delta t,x)(y^{\pm,\ell}_{k} (x)-x)\big]  \\[8pt]
\; & = & \varphi(t_k,x) + \lambda^{\pm,\ell}_k(x) \Delta t\bigg( \partial_t \varphi(t_k,x) +  D_x\varphi(t_k,x)^{\top} b(t_k,x)  \\[8pt]
\; & \; &+ \frac{p}{2}\sigma^{\ell}(t_k,x)^{\top} \left[ D_x^{2}\varphi(t_k,x)\sigma^{\ell}(t_k,x)\right]\bigg)\pm \sqrt{p\lambda^{\pm,\ell}_k(x)}D_x\varphi(t_k,x)^{\top} \sigma^{\ell}(t_k,x)\\[8pt]
\; & \; &  +\lambda^{\pm,\ell}_k(x)\Dt \Big(\epsilon_k^{1,\pm,\ell}(\Delta t,x) +\frac{(y_{k}^{\pm,\ell}(x)-x)}{\lambda^{\pm,\ell}_k(x)\Dt}^{\top}\epsilon_k^{2,\pm,\ell}(\Delta t,x) \\[8pt]
\; & \; & + \frac{(y^{\pm,\ell}_{k} (x)-x)}{\lambda^{\pm,\ell}_k(x)\Dt}^{\top}\big[\epsilon_{k}^{3,\pm,\ell}(\Delta t,x)(y^{\pm,\ell}_{k} (x)-x)\big]+\epsilon_k^{4,\pm,\ell}(\Delta t,x) \Big),
\end{array}
\end{equation}
where 
$$
\begin{array}{rcl}
\epsilon_{k}^{1,\pm,\ell}(\Delta t,x)&=&\int_{0}^{1}\left[\partial_{t} \varphi(t_{k}+ \xi \lambda^{\pm,\ell}_k(x)\Dt, x+\xi (y^{\pm,\ell}_{k} (x)-x))-\partial_{t}\varphi(t_k,x) \right]\dd \xi, \\[8pt]
\epsilon_{k}^{2,\pm,\ell}(\Delta t,x) &=&  \int_{0}^{1}\left[D_x\varphi(t_{k}+ \xi \lambda^{\pm,\ell}_k(x)\Dt, x+\xi(y^{\pm,\ell}_{k} (x)-x))- D_x\varphi(t_{k}, x+\xi(y^{\pm,\ell}_{k} (x)-x)) \right]\dd \xi, \\[8pt]
 \epsilon_{k}^{3,\pm,\ell}(\Delta t,x) &=& \int_{0}^{1}(1-\xi)\left[ D^2_{x}\varphi(t_k,x+\xi (y^{\pm,\ell}_{k} (x)-x))- D_x^2\varphi(t_k,x) \right] \dd \xi,\\[8pt]
 \epsilon_k^{4,\pm,\ell}(\Delta t,x)&=& \frac{b(t_k,x)}{2}^{\top}\big[D_x^2\varphi(t_k,x) b(t_k,x)\big]\lambda^{\pm,\ell}_k(x)\Delta t  \pm b(t_k,x)^{\top} \big[D_x^2\varphi(t_k,x)\sigma^\ell(t_k,x)\big] 
 (p\lambda^{\pm,\ell}_k(x)\Delta t)^{\frac{1}{2}}.
\end{array}
$$
It follows from~\eqref{sviluppo_dim_1} and~\eqref{weights_multi_dimensional_case} that~\eqref{somma_taylor_pm} holds, with
$$ \epsilon_k^{\pm,\ell}(\Delta t,x)= \epsilon_k^{1,\pm,\ell}(\Delta t,x) +\frac{(y_{k}^{\pm,\ell}(x)-x)}{\lambda^{\pm,\ell}_k(x)\Dt}^{\top} \epsilon_k^{2,\pm,\ell}(\Delta t,x)+ \frac{(y^{\pm,\ell}_{k} (x)-x)}{\lambda^{\pm,\ell}_k(x)\Dt}^{\top}\big[\epsilon_{k}^{3,\pm,\ell}(\Delta t,x)(y^{\pm,\ell}_{k} (x)-x)\big]+\epsilon_k^{4,\pm,\ell}(\Delta t,x).$$
Since 
$$\sup_{ k\in \mathcal{I}_{\Delta t}^*,\, x\in \Omega}\big\{| \epsilon_k^{1,\pm,\ell}(\Delta t,x)|\vee | \epsilon_k^{3,\pm,\ell}(\Delta t,x)|\vee | \epsilon_k^{4,\pm,\ell}(\Delta t,x)| \big\}  \to 0 \quad \text{as } \;  \Delta t\to 0$$
 and
 $\varphi\in \mathcal{F}_{\beta}(\overline{Q}_T)$, which implies that 
$$ \sup_{k\in \mathcal{I}_{\Delta t}^*, x\in \Omega}\bigg|\frac{(y_{k}^{\pm,\ell}(x)-x)}{\lambda^{\pm,\ell}_k(x)\Dt}\epsilon_k^{2,\pm,\ell}(\Delta t,x) \bigg| \to 0 \quad \text{as } \;  \Delta t\to 0,$$
we conclude that~\eqref{errore_globale_apendice} holds. 
\end{proof}
 
\bibliographystyle{plain}
\bibliography{BibDir}
\end{document}